%% file: main.tex
\documentclass[11pt]{article}

\PassOptionsToPackage{hypertexnames=false}{hyperref}

\input{./settings_latex/set_text_en_default.tex} % jp or en
\input{./settings_latex/set_ref_clever_en_simple.tex} % default or clever_jp or clever_en
\input{./settings_latex/set_math_symbol_simple.tex}
\input{./settings_latex/set_bib_default.tex} % use biblatex

\usepackage{tikz}
\usetikzlibrary{positioning,arrows.meta,fit,backgrounds}

\begin{document}

\title{\vspace{-10mm}\textbf{Asymptotic Error Distributions for Euler-Maruyama Schemes with Additive Semimartingale Noise}}
\author{Mikio Hirokane\thanks{Division of Mathematical Science for Social Systems, Department of Systems Innovation, Graduate School of Engineering Science, The University of Osaka. E-mail : \texttt{hirokane@sigmath.es.osaka-u.ac.jp}}}
\date{October 31, 2026}

\maketitle

\begin{abstract}
\begingroup
\setlength{\emergencystretch}{1em}
We establish a joint stable limit theorem for integrated discretization errors of finite-dimensional c\`adl\`ag semimartingales whose predictable characteristics are absolutely continuous in time and satisfy a pathwise integrability condition. The normalized integrated error, evaluated on the discretization grid, converges stably jointly with the discretized semimartingale in the vector-valued Skorokhod \(J_1\)-topology. No martingale representation property is assumed. We apply this theorem to the Euler--Maruyama scheme for stochastic differential equations with additive semimartingale noise. For a \(C^1\) drift with locally Lipschitz derivative and a nonexplosive solution, the grid-point error normalized by \(n\) converges stably jointly with the discretized driver. The limiting error solves a linear integral equation with an explicit inhomogeneous term involving the driver, an auxiliary Brownian motion, and independent uniform marks attached to the jumps. We also obtain the same limit under two alternative
sets of assumptions allowing less regular drifts, based on small-jump integrability or smoothing estimates for L\'evy semigroups. For continuous drivers, the continuous-time error converges stably in the uniform topology.
\par\endgroup
\end{abstract}

\tableofcontents

\section{Introduction}
\label{sec:introduction}

Limit theorems for numerical error processes form a classical part of the asymptotic theory of stochastic discretization.
Rootz\'en \cite{rootzen1980limit} studied limit distributions for discretization errors of stochastic integrals.
Kurtz and Protter \cite{kurtz1991weak} developed a general weak convergence theory for stochastic integrals and stochastic differential equations, while Jacod and Protter \cite{jacod1998asymptotic} studied asymptotic error distributions for Euler approximations.
This line of research has since developed in several
directions, including approximations of stochastic
integrals and SDEs on random time grids
\cite{fukasawa2011discretization,fukasawa2020efficient},
limit theorems associated with multilevel Monte Carlo
methods \cite{benalaya2015central,dereich2016multilevel},
fractional diffusions \cite{hu2016rate},
and stochastic Volterra equations
\cite{fukasawa2023limit,nualart2023error}.
A recurring theme in this literature is functional convergence of the normalized error process, often in a stable form that retains its joint asymptotic dependence on the original driving process.

In this paper, we establish a joint stable limit theorem for integrated discretization errors of semimartingales and apply it to the Euler--Maruyama scheme for stochastic differential equations with additive noise. Consider the following additive SDE on \([0,1]\):
\begin{equation}\label{eq:intro_sm_em_sde}
    X_t
    =
    x_0+\int_0^t f(X_{s-})\,ds+\sigma Y_t,
\end{equation}
where \(Y\) is a \(q\)-dimensional c\`adl\`ag semimartingale with
\(Y_0=0\), \(x_0\in\mathbb R^d\) is deterministic, and
\(\sigma\in\mathbb R^{d\times q}\) is a constant matrix.
For
\[
    \varphi_n(t)=\frac{\lfloor nt\rfloor}{n},
\]
the Euler--Maruyama approximation is
\begin{equation}\label{eq:intro_sm_em_scheme}
    \hat X_t^n
    =
    x_0+\int_0^t f(\hat X_{\varphi_n(s)}^n)\,ds+\sigma Y_t.
\end{equation}
Because the noise enters additively, the exact equation and its Euler approximation contain the same noise term \(\sigma Y_t\).
Thus the numerical error arises entirely from discretizing the drift integral.

For Brownian SDEs with state-dependent diffusion
coefficients, the classical Euler error is typically
of order \(n^{-1/2}\), and its normalized error process
admits a functional limit; see, in particular,
Jacod and Protter \cite{jacod1998asymptotic}.
When the diffusion coefficient is constant, the
leading \(\sqrt n\)-scale fluctuation vanishes,
and the order-\(n^{-1}\) error becomes relevant.
In the scalar additive setting, Yan
\cite[Section~6]{yan2005asymptotic}
obtains a weak functional limit at rate \(n\)
for a simplified Milstein scheme, which reduces
to the Euler scheme for constant diffusion.
Under suitable smoothness assumptions, stable
functional limits can also be obtained by
specializing the second-order expansion of
Podolskij, Veliyev and Yoshida
\cite{podolskij2020edgeworth}
and the result of Sauri
\cite[Theorem~1(2)]{sauri2025euler}\nopagebreak[4]
to additive Brownian SDEs without delay.
For stochastic Hamiltonian systems with additive
Brownian noise, Hong, Liang and Sheng
\cite[Proposition~2.6]{hong2026asymptotic}
obtain fixed-time error limits for a
\(\theta\)-method, including explicit Euler.

For equations with jumps, asymptotic Euler errors exhibit additional phenomena.
Jacod and Protter \cite{jacod1998asymptotic} and Jacod
\cite{jacod2004euler} obtain Euler error limits for jump and L\'evy-driven equations, with normalizations that may depend on the activity of the small jumps.
Wang \cite{wang2015euler} extended Jacod's analysis from pure-jump L\'evy drivers to a class of pure-jump It\^o semimartingales.
In these studies, the principal error terms arise from the discretization of stochastic integrals with state-dependent coefficients.
For the additive equation \eqref{eq:intro_sm_em_sde}, by contrast, the discretization error comes entirely from the drift integral.

More precisely, for a c\`adl\`ag semimartingale \(V\), put \(V_t^{[n]}=V_{\varphi_n(t)}\) and consider
\[
    \mathcal I_t^n(V)
    =
    n\int_0^{\varphi_n(t)}
       \{V_{s-}-V_{\varphi_n(s)}\}\,ds.
\]
Integrated discretization errors of this type were studied by Jacod, Jakubowski and M\'emin
\cite{jacod2003asymptotic}
and later by Jacod and Protter
\cite[Chapter~6]{jacod2012discretization}.
Jacod, Jakubowski and M\'emin establish tightness
for arbitrary semimartingales and identify a stable
limit under absolutely continuous characteristics
and a martingale representation property.
Their stable limit includes an auxiliary Gaussian
term and contributions from jumps carrying
independent uniform marks.
Jacod and Protter extend the limit theorem to smooth functions of multidimensional It\^o semimartingales.
Their continuous It\^o semimartingale result does not require a martingale representation property, whereas their discontinuous result is formulated under the stronger structural Assumption~(H).

Our first main result, \cref{thm:integrated_discretization_limit}, is a joint stable limit theorem for integrated discretization errors.
For an \(r\)-dimensional c\`adl\`ag semimartingale \(V\) whose predictable characteristics satisfy
\eqref{eq:app_first_companion_characteristics} and
\eqref{eq:app_first_companion_integrability}, it gives
\[
    (V^{[n]},\mathcal I^n(V))
    \Rightarrow^{\mathrm{st}}
    (V,\mathcal I(V))
\]
in \(D([0,1],\mathbb R^{2r})\) with the vector-valued Skorokhod
\(J_1\)-topology, stably with respect to the original sigma-field
\(\mathcal F\). The limiting process \(\mathcal I(V)\) is defined in
\eqref{eq:app_generic_first_companion_limit}.
\par
In particular, our theorem does not require the
martingale representation property assumed in
\cite[Theorem~1.4]{jacod2003asymptotic}.
The possibility of dispensing with this property was
already noted in \cite[Remark~(3), p.~597]{jacod2003asymptotic}.
Our theorem also avoids Assumption~(H) used for
discontinuous It\^o semimartingales in
\cite[Chapter~6]{jacod2012discretization}.
This assumption imposes additional regularity
requirements on the coefficients and, after
localization, requires a deterministic square-integrable upper bound for
\(|\delta|\wedge1\), where \(\delta\) is the jump coefficient
in the Grigelionis representation \cite[(4.4.1) and Assumption~4.4.2(iii)]{jacod2012discretization}.
Under these weaker assumptions, we identify the
Gaussian and marked jump components jointly in
the vector-valued \(J_1\)-topology.
As in \cite{jacod2012discretization}, the same
uniform mark is used for all coordinates of
each vector jump.

Our second main result, \cref{thm:sm_em_joint}, concerns the Euler--Maruyama error for the additive equation \eqref{eq:intro_sm_em_sde}. We assume
\[
    f\in C^1(\mathbb R^d;\mathbb R^d),
    \qquad
    Df\ \text{locally Lipschitz},
\]
and assume separately that the solution of
\eqref{eq:intro_sm_em_sde} is nonexplosive on \([0,1]\).
We further assume that the predictable characteristics of \(Y\) satisfy \eqref{eq:sm_ac_characteristics} and \eqref{eq:sm_ac_integrability}.

The theorem assumes nonexplosion directly, without
imposing a global growth condition on \(f\),
moment conditions on the large jumps of \(Y\),
or a rank condition on \(\sigma\).

Under additional conditions on the noise, we also obtain the same joint stable error limit for less regular drifts.
The first alternative in \cref{thm:sm_em_relaxed_regularity} applies to drivers without a continuous local martingale part and imposes a small-jump integrability condition adapted to the modulus of continuity of \(Df\).
The second applies to L\'evy drivers and combines local Sobolev regularity of \(f\) with a semigroup smoothing condition on the effective noise \(\sigma Y\).
The latter condition requires the time marginals of \(\sigma Y\) to admit densities on the state space.

Put
\[
    Y_t^{[n]}=Y_{\varphi_n(t)},
    \qquad
    \bar U_t^n
    =
    n\bigl(X_{\varphi_n(t)}-\hat X_{\varphi_n(t)}^n\bigr).
\]
Theorem~\ref{thm:sm_em_joint} states that
\[
    (Y^{[n]},\bar U^n)
    \Rightarrow^{\mathrm{st}}
    (Y,U)
\]
in \(D([0,1],\mathbb R^{q+d})\) with the
vector-valued Skorokhod \(J_1\)-topology,
stably with respect to \(\mathcal F\).
The limiting error is the unique c\`adl\`ag solution of
\[
    U_t
    =
    \mathcal Z_t+\int_0^t Df(X_{s-})U_s\,ds,
\]
where
\[
    \mathcal Z_t
    =
    \frac12\{f(X_t)-f(x_0)\}
    +
    \frac1{\sqrt{12}}
       \int_0^t Df(X_{s-})\sigma c_s^{1/2}\,dW_s
    +
    \mathcal K_t^f.
\]
Here \(c\) denotes the density of the continuous covariance characteristic of \(Y\), \(W\) is a \(q\)-dimensional Brownian motion defined on an extension of the original probability space and independent of \(\mathcal F\), and \(\mathcal K^f\) is the centered marked jump process defined in \eqref{eq:sm_centered_marked_forcing}.
If \((T_p)\) enumerates the nonzero jumps of \(Y\), then
\[
    \Delta U_{T_p}
    =
    \kappa_p
    \bigl\{
       f(X_{T_p-}+\sigma\Delta Y_{T_p})
       -
       f(X_{T_p-})
    \bigr\},
\]
where the \(\kappa_p\) are uniform on \([0,1]\), mutually independent and independent of \(W\) and \(\mathcal F\).
In particular, the jump contribution involves
the actual increment of \(f\) across a jump of
\(X\), and the same uniform mark is used for
all coordinates of that increment.
The form of these Gaussian and marked jump
terms is already known in the literature on
integrated discretization errors
\cite{jacod2003asymptotic,jacod2012discretization}.
Our contribution is to identify the joint stable
limit under weaker assumptions on the predictable
characteristics of discontinuous semimartingales
and to transfer this limit to the Euler error.

The use of the discretized processes \(Y^{[n]}\) and \(\bar U^n\) is essential when \(Y\) has jumps.
The exact solution and its Euler approximation have the same jumps, so
\(X-\hat X^n\) is continuous, whereas the limiting process \(U\) generally has jumps.
Consequently, the continuous-time error cannot converge in the \(J_1\)-topology to a discontinuous limit.
Sampling both the driver and the error on the same grid aligns their jumps associated with each jump of \(Y\), allowing joint convergence with a single time change for the whole vector.
This distinction between convergence in the vector-valued \(J_1\)-topology
and convergence in the product of the coordinatewise \(J_1\)-topologies
already appears in the Euler-error analysis of Jacod and Protter
\cite[Lemma~2.1 and Remark~3.1]{jacod1998asymptotic}.

The continuous case should be distinguished from
the discontinuous extension.
If \(Y\) is a continuous It\^o semimartingale
with absolutely continuous characteristics,
the corresponding Euler error limit for the
additive SDE can be obtained from classical results.
Under the local \(C^{1,1}\) assumption, a localization
and mollification argument shows that \(f(X)\) is a
semimartingale with absolutely continuous characteristics.
The continuous case of
\cite[Theorem~6.1.2]{jacod2012discretization}
can then be applied to the identity map on the
augmented semimartingale \((Y,f(X))\),
and the integrated-error limit can be propagated
through a linear integral equation.
For continuous drivers satisfying the assumptions of
\cref{thm:sm_em_joint} or
\cref{thm:sm_em_relaxed_regularity},
\cref{cor:sm_continuous_time_error}
gives stable convergence of the continuous-time
Euler error in the uniform topology.

Localization has been used to study asymptotic
Euler errors under local coefficient assumptions;
see, for example,
\cite{neuenkirch2009asymptotic,protter2020asymptotic}.
Reducing drift regularity is also an active topic
in Euler error limit theorems.
For uniformly elliptic Brownian SDEs, allowing
state-dependent diffusion coefficients,
Dareiotis, Gerencs\'er and L\^e
\cite{dareiotis2026central}
establish central limit theorems for
\(\sqrt n\)-normalized Euler errors under
substantially weaker drift regularity.
Altmeyer \cite{altmeyer2023central} establishes
related central limit theorems for discretized
occupation-time functionals of continuous It\^o
semimartingales, allowing test functions of low
regularity.
In our setting, the local \(C^{1,1}\) condition,
together with the additive structure,
allows us to establish a change-of-variable
formula for \(f(X)\) with absolutely continuous
semimartingale characteristics, without requiring
nondegeneracy of the noise.
Strong order-one estimates for the associated
random ordinary differential equation are also
available under stronger global regularity and
integrability assumptions \cite{kloeden2025strong}.
Our distributional result does not imply
corresponding moment bounds.
Since local \(C^{1,1}\) regularity alone does not
rule out finite-time explosion, nonexplosion is
assumed separately, with sufficient conditions
given in Remark~\ref{rem:sm_nonexplosion}.

Subsection~\ref{sec:regularity_refinements_examples}
gives a multidimensional example with a gradient
drift satisfying
\(f\in C^1\cap W_{\mathrm{loc}}^{2,p}\) but
\(f\notin C_{\mathrm{loc}}^{1,1}\).
The example applies to nondegenerate Brownian noise
and to invertible linear images of rotationally
symmetric \(\alpha\)-stable L\'evy noise with
\(1<\alpha<2\), demonstrating that condition (B)
of \cref{thm:sm_em_relaxed_regularity} genuinely
extends the local \(C^{1,1}\) assumption.
Overdamped Langevin dynamics are additive Brownian
SDEs, and their Euler discretization is the
unadjusted Langevin algorithm (ULA).
Joint stable limits of numerical errors are also
relevant to multilevel Monte Carlo methods, as
illustrated by Ben Alaya and Kebaier
\cite{benalaya2015central}, while MLMC methods
for invariant measures of Langevin dynamics
have been studied by Giles et al.
\cite{giles2020multilevel}.
The present results concern single-level Euler
errors over a fixed time horizon.

To illustrate the scope of
\cref{thm:integrated_discretization_limit},
Subsection~\ref{subsec:singular_marked_point_process}
constructs a two-dimensional marked point-process
driver based on a triangular Hawkes-type intensity
with a gamma kernel having an integrable singularity
at zero, motivated by \cite{hainaut2025rough}.
In the specified filtration, this driver satisfies
our characteristic assumptions but fails both
the martingale representation property relative
to the driver and Assumption~(H) of
\cite{jacod2012discretization}.
Thus the corresponding limit theorems of
\cite{jacod2003asymptotic,jacod2012discretization}
do not apply directly to this example.
Its application to a linear equation in
Subsection~\ref{sec:regularity_refinements_examples}
shows that jumps affecting both coordinates can
induce nonzero conditional covariance between
the limiting Euler error coordinates, because
the same auxiliary uniform mark is used for
both coordinates at each such jump.

The proof proceeds in two stages.
First, we prove \cref{thm:integrated_discretization_limit} in Appendix~\ref{app:first_integrated_companion}.
We approximate the underlying semimartingale by processes satisfying Assumption~(H), apply \cite[Theorem~6.1.2]{jacod2012discretization} to these approximations, and then remove the approximation using estimates uniform in the mesh parameter.
This argument identifies the continuous Gaussian and marked jump components jointly.
Second, in Section~\ref{sec:proof_main_theorem}, we transfer this limit to the Euler error through the augmented semimartingale
\[
    \Xi=(Y,f(X)-f(x_0)).
\]
The required absolute continuity of the characteristics of \(f(X)\) is obtained by localization and mollification under the \(C^{1,1}_{\mathrm{loc}}\) assumption.
For the relaxed regularity results,
the criterion in \cref{prop:sm_euler_transfer}
is verified either through summability of
small-jump Taylor remainders or through a
Sobolev It\^o formula using occupation estimates
obtained from L\'evy semigroup bounds.
Consistency of the Euler scheme and continuity
of the corresponding linear integral equation
then transfer the integrated-error limit to the
normalized Euler error.

The paper is organized as follows.
Section~\ref{sec:integrated_discretization_errors}
presents the general integrated discretization
theorem and a point-process example illustrating
the relaxation of the assumptions in earlier
limit theorems.
Section~\ref{sec:setting_main_result} presents the Euler--Maruyama error
theorem, the two alternatives allowing weaker drift regularity, and examples.
Section~\ref{sec:proof_main_theorem} proves the Euler error results by
applying the general theorem to the augmented semimartingale.
\par
\begin{samepage}
Appendix~\ref{app:deterministic_localization_propagation} collects the
consistency estimates and the stability result for linear integral
equations.
Appendix~\ref{app:first_integrated_companion} proves the general
integrated discretization limit by regularization and approximation.
Appendix~\ref{app:levy_smoothing}
contains the occupation estimates, a Sobolev
It\^o formula, and the localization arguments
used to prove
\cref{thm:sm_em_relaxed_regularity}
under condition~(B).
\par
\end{samepage}

\section[Integrated Discretization Errors]{Integrated Discretization Errors}
\label{sec:integrated_discretization_errors}

Let \((\Omega,\mathcal F,\mathbb F,\mathbb P)\), with
\(\mathbb F=(\mathcal F_t)_{0\le t\le1}\), satisfy the usual conditions.

The notation \(|\cdot|\) denotes the Euclidean norm for vectors and matrices; \(\|\cdot\|_{\mathrm{op}}\) denotes
the operator norm, and \(\|g\|_\infty=\sup_{t\le1}|g_t|\).
All Skorokhod spaces carry the usual \(J_1\)-topology, with one common
time change for all coordinates of a vector-valued path.

Set
\begin{equation}\label{eq:sm_phi_def}
    \varphi_n(t)
    \coloneqq
    \frac{\lfloor nt\rfloor}{n},
    \qquad
    0\le t\le1.
\end{equation}

\subsection[A joint stable limit theorem]{A joint stable limit theorem}
\label{subsec:general_integrated_discretization_limit}

We present a joint stable limit theorem for integrated discretization
errors of finite-dimensional semimartingales with absolutely continuous
predictable characteristics. No martingale representation property or
Assumption~(H) of \cite{jacod2012discretization} is imposed.

For a finite-dimensional c\`adl\`ag process \(V\), set
\begin{align}
    V_t^{[n]}
    &\coloneqq
    V_{\varphi_n(t)},
    \label{eq:app_generic_grid_process}\\
    \mathcal I_t^n(V)
    &\coloneqq
    n\int_0^{\varphi_n(t)}
    \bigl(V_{s-}-V_{\varphi_n(s)}\bigr)\,ds.
    \label{eq:app_generic_first_companion}
\end{align}
Replacing \(V_{s-}\) by \(V_s\) changes no Lebesgue integral.

Let \(V=(V^1,\ldots,V^r)\), \(r\ge1\), be a c\`adl\`ag
\(\mathbb F\)-semimartingale. Fix a bounded Borel truncation function
\(h:\mathbb R^r\to\mathbb R^r\) equal to the identity near zero.
Let \(\mu^V\) be the jump measure of \(V\) and \(\nu^V\) its predictable
compensator. Assume that the predictable characteristics relative to
\(h\) are
\begin{equation}\label{eq:app_first_companion_characteristics}
    B_t^{V}
    =
    \int_0^t b_s\,ds,
    \qquad
    C_t^{V}
    =
    \int_0^t c_s\,ds,
    \qquad
    \nu^{V}(ds,dz)
    =
    F_s(dz)\,ds.
\end{equation}
Here \(b\) is predictable, \(c\) is predictable and symmetric
nonnegative, and \(F\) is a predictable L\'evy kernel, with
\begin{equation}\label{eq:app_first_companion_integrability}
    \int_0^1
    \left(
        |b_s|
        +
        \operatorname{tr}(c_s)
        +
        \int_{\mathbb R^r}(1\wedge|z|^2)F_s(dz)
    \right)ds
    <\infty
    \qquad\text{a.s.}
\end{equation}

Choose stopping times \((T_p)\) with disjoint graphs whose finite values
enumerate exactly the nonzero jump times of \(V\).  On the very good extension of
\cite[(4.1.16)]{jacod2012discretization}, take i.i.d. uniform \([0,1]\)
marks \((\kappa_p)\) and an \(r\)-dimensional Brownian motion \(W'\),
mutually independent and independent of \(\mathcal F\).
Each mark is revealed at its jump time.
The extended filtration makes \(W'\) adapted and preserves original
martingales and semimartingale characteristics. Write
\(\widetilde{\mathbb P}\) and \(\widetilde{\mathbb E}\) for probability
and expectation on this extension.
The marked measure and compensator are
\begin{equation}\label{eq:app_generic_marked_jump_measure}
    \widetilde\mu^{V}(ds,dz,du)
    =
    \sum_{p:T_p\le1}
    \delta_{(T_p,\Delta V_{T_p},\kappa_p)}(ds,dz,du).
\end{equation}
\begin{equation}\label{eq:app_generic_marked_jump_compensator}
    \widetilde\nu^{V}(ds,dz,du)
    =
    \nu^{V}(ds,dz)\,du
    =
    F_s(dz)\,ds\,du.
\end{equation}

Writing \(V^c\) for the continuous local martingale part and
\(c^{1/2}\) for the predictable symmetric square root, set
\begin{align}
    \mathcal I_t(V)
    &\coloneqq
    \frac12 B_t^{V}
    +
    \frac12 V_t^c
    +
    \frac1{\sqrt{12}}
    \int_0^t c_s^{1/2}\,dW_s'
    \notag\\
    &\quad
    +
    \int_0^t\int_{\mathbb R^r}\int_0^1
    u h(z)
    (\widetilde\mu^{V}-\widetilde\nu^{V})(ds,dz,du)
    \notag\\
    &\quad
    +
    \int_0^t\int_{\mathbb R^r}\int_0^1
    u\{z-h(z)\}
    \widetilde\mu^{V}(ds,dz,du).
    \label{eq:app_generic_first_companion_limit}
\end{align}
The Brownian integral and the compensated marked integral are well defined
as locally square-integrable local martingales after localization.
Since \(h\) is bounded and agrees with the identity in a neighborhood of
zero, there exists a constant \(C_h<\infty\) such that
\[
    |h(z)|^2\le C_h(1\wedge|z|^2),\qquad z\in\mathbb R^r.
\]
Together with \(|c_s^{1/2}|^2=\operatorname{tr}(c_s)\), this bound and
\eqref{eq:app_first_companion_integrability} give locally finite
predictable brackets.
The uncompensated integral is a finite sum because \(z-h(z)\) vanishes
near zero.  Define also
\begin{align}
    {\mathcal K}_t(V)
    &\coloneqq
    \int_0^t\int_{\mathbb R^r}\int_0^1
    \left(u-\frac12\right)h(z)
    (\widetilde\mu^{V}-\widetilde\nu^{V})(ds,dz,du)
    \notag\\
    &\quad
    +
    \int_0^t\int_{\mathbb R^r}\int_0^1
    \left(u-\frac12\right)\{z-h(z)\}
    \widetilde\mu^{V}(ds,dz,du).
    \label{eq:app_centered_marked_companion}
\end{align}

Linearity of the compensated integrals and the canonical representation
of \(V\) relative to \(h\) give
\begin{equation}\label{eq:app_generic_first_companion_compact}
    \mathcal I_t(V)
    =
    \frac12\{V_t-V_0\}
    +
    \frac1{\sqrt{12}}
    \int_0^t c_s^{1/2}\,dW_s'
    +
    {\mathcal K}_t(V).
\end{equation}
On \(\{T_p\le1\}\), since the compensator has no time atoms,
\begin{equation}\label{eq:app_generic_companion_jump}
    \Delta\mathcal I(V)_{T_p}
    =
    \kappa_p\Delta V_{T_p}.
\end{equation}

\begin{theorem}[Joint stable limit for integrated discretization errors]
\label{thm:integrated_discretization_limit}
Under \eqref{eq:app_first_companion_characteristics} and
\eqref{eq:app_first_companion_integrability},
\begin{equation}\label{eq:app_unweighted_first_companion_limit}
    \left(
        V^{[n]},
        \mathcal I^n(V)
    \right)
    \Rightarrow^{\mathrm{st}}
    \left(
        V,
        \mathcal I(V)
    \right)
\end{equation}
in \(D([0,1],\mathbb R^{2r})\)
endowed with the usual vector-valued Skorokhod \(J_1\)-topology.  The stable
convergence is with respect to the original sigma-field \(\mathcal F\).  The
limit may be realized on the same very good filtered extension used in
\eqref{eq:app_generic_first_companion_limit}.
\end{theorem}

Here stable convergence means that, for every bounded
\(\mathcal F\)-measurable \(\xi\) and bounded continuous \(\Phi\) on
\(D([0,1],\mathbb R^{2r})\),
\begin{equation}\label{eq:sm_stable_test_function_form}
    \mathbb E\!\left[
        \xi\,\Phi(V^{[n]},\mathcal I^n(V))
    \right]
    \longrightarrow
    \widetilde{\mathbb E}\!\left[
        \xi\,\Phi(V,\mathcal I(V))
    \right].
\end{equation}
See \cite{jacod2012discretization} for background on stable
convergence.

The proof in Appendix~\ref{app:first_integrated_companion} regularizes
the characteristics, applies \cite[Theorem~6.1.2]{jacod2012discretization}
to the regularized processes, and removes the regularization using
estimates uniform in the mesh parameter.\looseness=-1

\subsection[A marked point process with singular intensity]{A marked point process with singular intensity}
\label{subsec:singular_marked_point_process}

We construct a two-dimensional marked point process that satisfies the
assumptions of \cref{thm:integrated_discretization_limit}. In the filtration
specified below, the martingale representation property relative to this
process fails, and the process does not satisfy Assumption~(H) of
\cite{jacod2012discretization}.
Let \(N^2\) be a Poisson process with rate \(\mu_2>0\).
Define the intensity by
\[
 \lambda_t=\mu_1+\eta\int_{[0,t)}k_\gamma(t-s)\,dN_s^2,
 \qquad {k_\gamma(u)=\frac{\delta^\gamma}{\Gamma(\gamma)}
                 u^{\gamma-1}e^{-\delta u},\quad u>0,}
\]
where \(0<\gamma<1\) and \(\mu_1,\eta,\delta>0\).
The Gamma kernel \(k_\gamma\) is of the form considered
in the rough Hawkes model of \cite{hainaut2025rough}.
Unlike that self-exciting model, the intensity
\(\lambda\) here depends on the exogenous Poisson process
\(N^2\).
Let \(\Pi\) be a Poisson random measure on
\([0,1]\times(0,\infty)\) with intensity \(ds\,du\),
independent of \(N^2\), and define
\[
 N_t^1
 =\int_0^t\int_0^\infty
 \mathbf1_{\{u\le\lambda_s\}}\Pi(ds,du).
\]
Conditionally on \(N^2\), the process \(N^1\) is an
inhomogeneous Poisson process with intensity \(\lambda\).
At each jump time \(T_j\) of \(N^1\), attach an independent mark
\(Z_j\) with law
\[
 \pi=p_{10}\delta_{(1,0)}+p_{01}\delta_{(0,1)}
                         +p_{11}\delta_{(1,1)},\qquad
 p_{10},p_{01},p_{11}>0,\quad p_{10}+p_{01}+p_{11}=1,
\]
and set \(Y_t=\sum_{T_j\le t}Z_j\).
The marks \(Z_j\) are mutually independent and independent
of \(N^2\) and \(\Pi\).
Work with the usual augmentation of the filtration generated by \(N^2\) and the marked point process.  Thus the filtration includes the exogenous process
\(N^2\), which determines the intensity of \(N^1\).
Since \(\int_0^\infty k_\gamma(u)\,du=1\),
\[
 \int_0^1\lambda_tdt\le\mu_1+\eta N_1^2<\infty\quad\text{a.s.},
 \qquad F_t^Y(dz)=\lambda_t\pi(dz),\quad
 b_t^Y=\lambda_t\int h(z)\pi(dz),\quad c_t^Y=0.
\]
Thus \eqref{eq:app_first_companion_characteristics} and
\eqref{eq:app_first_companion_integrability} hold with \(V=Y\) and \(r=2\).
There are almost surely no simultaneous jumps of \(N^1\) and \(N^2\).
The process \(N_t^2-\mu_2t\) is a martingale in this filtration.
Its jumps occur at times when \(Y\) does not jump,
so it cannot be represented using
\(Y^c=0\) and \(\mu^Y-\nu^Y\).
Hence the martingale representation property with respect to
\(Y\) fails in this filtration.

Assumption (H) also fails.  Its bound on the jump coefficient
\cite[Assumption~4.4.2(iii)]{jacod2012discretization} would imply,
after localization, the existence of stopping times
\(\tau_j\uparrow1\) almost surely and finite deterministic constants
\(C_j\) such that
\[
 F_t^Y(\{|z|\ge1\})\le C_j,
 \qquad dt\otimes d\mathbb P\text{-a.e. on }\{t\le\tau_j\}.
\]
Here \(F_t^Y(\{|z|\ge1\})=\lambda_t\).
Let \(S\) be the first jump time of \(N^2\).
On \(\{S<1\}\),
\[
 \lambda_{S+u}\ge\eta k_\gamma(u)
 \longrightarrow\infty
 \qquad(u\downarrow0).
\]
By Fubini's theorem, for almost every path the localized bound
holds for almost every \(t\le\tau_j\), for every \(j\).
If \(\tau_j>S\) on \(\{S<1\}\), the displayed lower bound
violates it on a nonempty interval immediately after \(S\).
It follows that \(\tau_j\le S\) on \(\{S<1\}\)
for every \(j\), almost surely.
Since \(\mathbb P(S<1)>0\), this contradicts
\(\tau_j\uparrow1\) almost surely.

\section[Euler--Maruyama Error Limits]{Euler--Maruyama Error Limits}
\label{sec:setting_main_result}
\label{sec:additive_semimartingale_em}

We apply the general theorem of Section~\ref{sec:integrated_discretization_errors}
to the Euler--Maruyama error for additive semimartingale noise on \([0,1]\).
We retain the probability space and common notation of that section.
The driver \(Y\) below plays the role of \(V\), with dimension \(r=q\);
we restate its characteristic and integrability assumptions for the Euler
error results.

We use Einstein's summation convention, with state indices \(i,j,k\)
ranging from \(1\) to \(d\) and driver indices \(\alpha,\beta,\gamma\)
from \(1\) to \(q\).

{%
We write \(f_j^i=\partial f^i/\partial x_j\) and
\(Df(x)=(f_j^i(x))_{i,j}\) for the Jacobian matrix.
When \(f\in C^2(\mathbb R^d;\mathbb R^d)\), we also write
\(f_{jk}^i=\partial^2f^i/(\partial x_j\partial x_k)\) and
\(D^2f^i(x)=(f_{jk}^i(x))_{j,k}\) for the Hessian of \(f^i\),
and \(D^2f\) for the collection of these Hessians.
}
The matrix \(\sigma=(\sigma_\alpha^i)\in\mathbb R^{d\times q}\) is
deterministic, with \((\sigma z)^i=\sigma_\alpha^iz^\alpha\).

\subsection{The equation and the Euler--Maruyama scheme}

Let \(Y\) be a \(q\)-dimensional c\`adl\`ag \(\mathbb F\)-semimartingale
with \(Y_0=0\).  Fix a bounded Borel truncation function
\(h:\mathbb R^q\to\mathbb R^q\) equal to the identity near zero.
For background on semimartingale characteristics and truncation
functions, see \cite{jacod2003limit}.
We assume that the predictable characteristics of \(Y\) relative to \(h\)
are absolutely continuous in time:
\begin{equation}\label{eq:sm_ac_characteristics}
    B_t^Y
    =
    \int_0^t b_s\,ds,
    \qquad
    C_t^Y
    =
    \int_0^t c_s\,ds,
    \qquad
    \nu^Y(ds,dz)
    =
    F_s(dz)\,ds.
\end{equation}
Let \(\mu^Y\) denote the jump measure of \(Y\), with predictable
compensator \(\nu^Y\).
Here \(b\) is predictable, \(c\) is predictable and symmetric nonnegative,
and \(F\) is a predictable L\'evy kernel.  We impose
\begin{equation}\label{eq:sm_ac_integrability}
    \int_0^1
    \left(
        |b_s|
        +
        \operatorname{tr}(c_s)
        +
        \int_{\mathbb R^q}(1\wedge|z|^2)F_s(dz)
    \right)ds
    <\infty
    \qquad\text{a.s.}
\end{equation}
No martingale representation property is assumed.

For an initial value \(x_0\in\mathbb R^d\), consider
\begin{equation}\label{eq:sm_em_sde}
    X_t^i
    =
    x_0^i
    +
    \int_0^t f^i(X_{s-})\,ds
    +
    \sigma_\alpha^iY_t^\alpha,
    \qquad
    1\le i\le d,
\end{equation}
and its Euler--Maruyama approximation
\begin{equation}\label{eq:sm_em_scheme}
    \hat X_t^{n,i}
    =
    x_0^i
    +
    \int_0^t f^i(\hat X_{\varphi_n(s)}^n)\,ds
    +
    \sigma_\alpha^iY_t^\alpha.
\end{equation}

For \cref{thm:sm_em_joint}, we assume
{%
\begin{equation}\label{eq:sm_drift_assumptions}
    f\in C^1(\mathbb R^d;\mathbb R^d),
    \qquad Df\text{ is locally Lipschitz}.
\end{equation}
We also denote this regularity by \(f\in C^{1,1}_{\mathrm{loc}}\).
Alternative assumptions are given in \cref{thm:sm_em_relaxed_regularity}.
}
In particular, \(f\) is locally Lipschitz.  Setting \(R=X-\sigma Y\) gives
the pathwise ordinary differential equation
\(R_t=x_0+\int_0^t f(R_s+\sigma Y_{s-})\,ds\), to which
Carath\'eodory theory applies and yields local existence and uniqueness.
Throughout, we assume that this solution, and hence the solution of
\eqref{eq:sm_em_sde}, exists on \([0,1]\) almost surely.
The Euler scheme is uniquely defined on the full interval
by recursion over the mesh cells.

\begin{remark}[Sufficient conditions for nonexplosion]
\leavevmode\phantomsection
\label{rem:sm_nonexplosion}
Suppose that \(f\) is locally Lipschitz.
Either the linear growth condition
\(|f(x)|\le K(1+|x|)\) for all \(x\in\mathbb R^d\)
and some \(K\ge0\), or the global one-sided
Lipschitz condition
\[
    \langle x-y,f(x)-f(y)\rangle
    \le L|x-y|^2,
    \qquad x,y\in\mathbb R^d,
\]
for some \(L\in\mathbb R\), ensures nonexplosion
for any c\`adl\`ag \(Y\).
Indeed, \(R=X-\sigma Y\) satisfies a pathwise
ordinary differential equation, and the conclusion
follows from Gronwall's inequality.

Under \eqref{eq:sm_ac_characteristics} and
\eqref{eq:sm_ac_integrability}, nonexplosion also
follows if, for some \(K\ge0\),
\begin{equation}\label{eq:sm_nonexplosion_radial}
    2\langle x,f(x)\rangle
    \le K(1+|x|^2),
    \qquad x\in\mathbb R^d.
\end{equation}
To see this, change the truncation to
\(h_0(z)=z\mathbf1_{\{|z|\le1\}}\), and denote
the corresponding drift density by \(b^{(0)}\).
Set
\[
    \widetilde Y_t
    =
    Y_t-\sum_{0<s\le t}
    \Delta Y_s\mathbf1_{\{|\Delta Y_s|>1\}},
\]
and let \(\widetilde X\) solve the equation
driven by \(\widetilde Y\).
It\^o's formula for \(V(x)=1+|x|^2\) shows that
the drift rate in the semimartingale decomposition
of \(V(\widetilde X)\) is bounded above by
\(a_sV(\widetilde X_{s-})\), where
\[
    a_s
    =K+\|\sigma\|_{\mathrm{op}}|b_s^{(0)}|
    +\|\sigma\|_{\mathrm{op}}^2
    \left(
        \operatorname{tr}(c_s)
        +\int_{|z|\le1}|z|^2F_s(dz)
    \right).
\]
By \eqref{eq:sm_ac_integrability},
\(A_1:=\int_0^1a_s\,ds<\infty\) almost surely.
For \(\rho>|x_0|\), let
\[
    \tau_\rho
    =
    \inf\{t\in[0,1]:
           |\widetilde X_t|\ge\rho\},
    \qquad
    \inf\varnothing=\infty,
\]
and put \(A_t=\int_0^t a_s\,ds\).
Then
\[
    e^{-A_{t\wedge\tau_\rho}}
    V(\widetilde X_{t\wedge\tau_\rho})
\]
is a nonnegative local supermartingale.
Consequently, for every \(m>0\),
\[
    \mathbb P(\tau_\rho\le1,\ A_1\le m)
    \le
    \frac{e^mV(x_0)}{1+\rho^2}.
\]
Letting \(\rho\to\infty\) and then
\(m\to\infty\) excludes explosion; compare
\cite[Theorem~3.5]{khasminskii2012stochastic}.
Finally, the removed jumps are almost surely
finite in number by
\eqref{eq:sm_ac_integrability}.
By localizing the state at a stopping time to
a bounded set, the same argument applies to
the restarted equation.
Thus these jumps can be restored by interlacing;
see \cite[Section~6.2]{applebaum2009levy}.
\end{remark}
\par

The exact solution and its Euler approximation have the same jumps,
\begin{equation}\label{eq:sm_common_jumps}
    \Delta X_t
    =
    \sigma\Delta Y_t
    =
    \Delta\hat X_t^n.
\end{equation}
Their difference satisfies
\begin{equation}\label{eq:sm_em_error_identity}
    X_t-\hat X_t^n
    =
    \int_0^t
    \{f(X_{s-})-f(\hat X_{\varphi_n(s)}^n)\}\,ds.
\end{equation}
Thus \(X-\hat X^n\) has absolutely continuous paths.  Define
\begin{equation}\label{eq:sm_grid_driver_error}
    Y_t^{[n]}
    \coloneqq
    Y_{\varphi_n(t)},
    \qquad
    \hat U_t^n
    \coloneqq
    n(X_t-\hat X_t^n),
    \qquad
    \bar U_t^n
    \coloneqq
    \hat U_{\varphi_n(t)}^n.
\end{equation}
The main theorem concerns the step processes \(Y^{[n]}\) and
\(\bar U^n\).
When the driver \(Y\) is continuous, we also prove
stable convergence of the continuous-time error
\(\hat U^n\) in the uniform topology.

\subsection{The limiting error process and the main theorem}

Let \(Y^c\) be the continuous local martingale part of \(Y\).
Choose stopping times \((T_p)\) with disjoint graphs whose finite values
enumerate exactly the nonzero jump times of \(Y\), as in
Subsection~\ref{subsec:general_integrated_discretization_limit}.
Use the corresponding very good filtered extension for \(Y\), with a
\(q\)-dimensional Brownian motion \(W\) and i.i.d. uniform \([0,1]\)
marks \((\kappa_p)\), mutually independent and independent of \(\mathcal F\).
The extended filtration makes \(W\) adapted and reveals each \(\kappa_p\)
at \(T_p\). We again write \(\widetilde{\mathbb P}\) and
\(\widetilde{\mathbb E}\) for probability and expectation on this extension.
The marked jump measure and its compensator are
\begin{equation}\label{eq:sm_marked_jump_measure}
    \widetilde\mu^Y(ds,dz,du)
    \coloneqq
    \sum_{p:T_p\le1}
    \delta_{(T_p,\Delta Y_{T_p},\kappa_p)}(ds,dz,du).
\end{equation}
\begin{equation}\label{eq:sm_marked_jump_compensator}
    \widetilde\nu^Y(ds,dz,du)
    =
    \nu^Y(ds,dz)\,du
    =
    F_s(dz)\,ds\,du.
\end{equation}
The same \(W\) and marks \((\kappa_p)\) are used for all coordinates of the
limit.

Let \(c^{1/2}\) denote the predictable nonnegative symmetric square root
of \(c\), and set
\begin{align}
    \Delta_f(x,z)
    &\coloneqq
    f(x+\sigma z)-f(x),
    \label{eq:sm_jump_displacement}\\
    \Theta_f(x,z)
    &\coloneqq
    \Delta_f(x,z)-Df(x)\sigma z.
    \label{eq:sm_theta_def_section}
\end{align}
Then \(\Delta_f(X_{s-},\Delta Y_s)=f(X_s)-f(X_{s-})\).

Fix a truncation function \(h_f:\mathbb R^d\to\mathbb R^d\) and define
\begin{align}
    \mathcal K_t^f
    &\coloneqq
    \int_0^t\int_{\mathbb R^q}\int_0^1
    \left(u-\frac12\right)
    h_f\!\left(\Delta_f(X_{s-},z)\right)
    (\widetilde\mu^Y-\widetilde\nu^Y)(ds,dz,du)
    \notag\\
    &\quad
    +
    \int_0^t\int_{\mathbb R^q}\int_0^1
    \left(u-\frac12\right)
    \left\{
        \Delta_f(X_{s-},z)
        -
        h_f\!\left(\Delta_f(X_{s-},z)\right)
    \right\}
    \widetilde\mu^Y(ds,dz,du).
    \label{eq:sm_centered_marked_forcing}
\end{align}

Define \(\mathcal Z\) by
\begin{equation}\label{eq:sm_em_forcing_limit}
    \mathcal Z_t
    \coloneqq
    \frac12\{f(X_t)-f(x_0)\}
    +
    \frac1{\sqrt{12}}
    \int_0^tDf(X_{s-})\sigma c_s^{1/2}\,dW_s
    +
    \mathcal K_t^f.
\end{equation}
These stochastic integrals are already well defined when \(f\in C^1\),
under \eqref{eq:sm_ac_characteristics}, \eqref{eq:sm_ac_integrability},
and nonexplosion.  Indeed, after localization to \(|X_{s-}|\le R\),
local boundedness of \(Df\) gives
\[
 |h_f(\Delta_f(X_{s-},z))|^2\le C_R(1\wedge|z|^2),\qquad
 |Df(X_{s-})\sigma c_s^{1/2}|^2\le C_R\operatorname{tr}(c_s).
\]
Moreover, \(\Delta_f-h_f(\Delta_f)\) vanishes for all sufficiently small
\(|z|\), uniformly before the stopping time.  Thus the compensated
integral has locally finite predictable bracket and the uncompensated
integral is a finite pathwise sum.  Estimates for the second-order
remainder \(\Theta_f\) are needed later for the change-of-variable
formulas, not for these definitions.
On \(\{T_p\le1\}\), since \(\widetilde\nu^Y\) has no time atoms,
\[
    \Delta\mathcal K^f_{T_p}
    =\left(\kappa_p-\frac12\right)\Delta_f(X_{T_p-},\Delta Y_{T_p}).
\]
Hence
\begin{equation}\label{eq:sm_forcing_jump}
    \Delta\mathcal Z_{T_p}
    =
    \kappa_p\Delta_f(X_{T_p-},\Delta Y_{T_p})
    =
    \kappa_p\{f(X_{T_p})-f(X_{T_p-})\}.
\end{equation}
{%
Under the hypotheses of \cref{thm:sm_em_joint} or
\cref{thm:sm_em_relaxed_regularity}, the process \(\mathcal Z\) does not
depend on the choice of \(h_f\), by
\eqref{eq:app_centered_remainder_identification} in
\cref{prop:sm_augmented_forcing_limit}.  Under the additional assumption
\(f\in C^2\), it also admits the representation \eqref{eq:sm_em_forcing_expanded}.
}

Let \(U\) be the unique c\`adl\`ag solution of
\begin{equation}\label{eq:sm_em_limit_equation}
    U_t^i
    =
    \mathcal Z_t^i
    +
    \int_0^t f_j^i(X_{s-})U_s^j\,ds,
    \qquad
    1\le i\le d.
\end{equation}
The equation is well posed pathwise because \(Df(X_-)\) is bounded on
\([0,1]\) along each nonexplosive path.  Its integral term is continuous,
so
\begin{equation}\label{eq:sm_limit_error_jump}
    \Delta U_t=\Delta\mathcal Z_t.
\end{equation}

\begin{theorem}[Euler--Maruyama error limit for additive semimartingale drivers]
\leavevmode\phantomsection % Keep the PDF target with the theorem after a page break.
\label{thm:sm_em_joint}
Assume \eqref{eq:sm_ac_characteristics},
\eqref{eq:sm_ac_integrability}, and \eqref{eq:sm_drift_assumptions}, and
assume that the solution of \eqref{eq:sm_em_sde} exists on \([0,1]\)
almost surely.  Then
\begin{equation}\label{eq:sm_em_joint_convergence}
    (Y^{[n]},\bar U^n)
    \Rightarrow^{\mathrm{st}}
    (Y,U)
\end{equation}
in \(D([0,1],\mathbb R^{q+d})\), endowed with the usual vector-valued
Skorokhod \(J_1\)-topology.  The stable convergence is with respect to the
original sigma-field \(\mathcal F\).  The limit is defined on the very good
filtered extension carrying the Brownian motion \(W\) and the uniform marks in
\eqref{eq:sm_marked_jump_measure}.
\end{theorem}

The proof of \cref{thm:sm_em_joint} is given in
Subsection~\ref{subsec:linear_propagation_main_theorem}.

\subsection[Weaker Drift Regularity and Examples]{Weaker Drift Regularity and Examples}
\label{sec:regularity_refinements_examples}
% Compatibility aliases for the regularity refinements and examples.
\label{sec:brownian_case}
\label{sec:discontinuous_examples}

The local \(C^{1,1}\) assumption in \cref{thm:sm_em_joint} is a
sufficient regularity condition for general semimartingale drivers.
The following result gives two alternative sufficient conditions.
Both yield the same limiting process and the same joint convergence in the vector-valued \(J_1\)-topology.

For \(f\in C^1(\mathbb R^d;\mathbb R^d)\) and \(R>0\), set
\begin{equation}\label{eq:sm_drift_modulus}
 \omega_R(r)=\sup\bigl\{\|Df(x)-Df(y)\|_{\mathrm{op}}:
       |x|,|y|\le R+2,\ |x-y|\le r\bigr\},\qquad r\ge0.
\end{equation}
We also write \(\varepsilon_\sigma=(1+\|\sigma\|_{\mathrm{op}})^{-1}\).
All Sobolev spaces below are defined componentwise with respect to
Lebesgue measure on the state space \(\mathbb R^d\).
For a Sobolev function, \(D^2f\) denotes its componentwise weak Hessian;
the classical notation \(f_{jk}^i\) is retained only where classical
second derivatives are assumed.

\begin{theorem}[Relaxed drift regularity]
\label{thm:sm_em_relaxed_regularity}
Assume \eqref{eq:sm_ac_characteristics} and
\eqref{eq:sm_ac_integrability}.  Suppose that
\(f\in C^1(\mathbb R^d;\mathbb R^d)\) and that the solution of
\eqref{eq:sm_em_sde} exists on \([0,1]\) almost surely.
In place of \eqref{eq:sm_drift_assumptions}, assume either of the
following conditions.
\begin{enumerate}
\item[(A)] \emph{Small-jump integrability.}  We have \(c=0\),
\(ds\otimes d\mathbb P\)-almost everywhere, and, for every \(R>0\),
\begin{equation}\label{eq:sm_jump_modulus_condition}
 \int_0^1\int_{|z|\le\varepsilon_\sigma}
       |z|\,\omega_R(\|\sigma\|_{\mathrm{op}}|z|)
       F_s(dz)\,ds<\infty\qquad\text{a.s.}
\end{equation}
\item[(B)] \emph{L\'evy transition-density estimates.}  The driver \(Y\) is an
\(\mathbb F\)-L\'evy process: its increments after time \(s\) are
independent of \(\mathcal F_s\).  For some \(p\in(1,\infty)\),
\(f\in W^{2,p}_{\mathrm{loc}}(\mathbb R^d;\mathbb R^d)\).
Writing \(p'=p/(p-1)\), the law of \(L_t=\sigma Y_t\) has a density
\(k_t\), for every \(0<t\le1\), satisfying
\begin{equation}\label{eq:levy_smoothing_K}
 \begin{gathered}
 k_t\in W^{1,1}(\mathbb R^d)\cap L^{p'}(\mathbb R^d),\qquad 0<t\le1,\\
 \int_0^1\left(\|k_t\|_{L^{p'}}+
                     \|\nabla k_t\|_{L^1}\right)dt<\infty.
 \end{gathered}
\end{equation}
The derivatives in this condition are weak derivatives.
\end{enumerate}
Then the conclusion of \cref{thm:sm_em_joint} remains valid, with
\(\mathcal K^f\), \(\mathcal Z\), and \(U\) given by
\eqref{eq:sm_centered_marked_forcing}, \eqref{eq:sm_em_forcing_limit},
and \eqref{eq:sm_em_limit_equation}, respectively.
\end{theorem}

The proof is given in
Subsection~\ref{subsec:proof_relaxed_regularity}.
Condition (A) concerns the continuous local martingale characteristic,
not the continuous finite-variation part: an absolutely continuous drift
of \(Y\) is still allowed.  Condition (B) is imposed on the effective
noise \(\sigma Y\), not merely on the untransformed driver \(Y\).

\begin{corollary}[Continuous-time convergence]
\label{cor:sm_continuous_time_error}
\label{thm:brownian_continuous_time_error}
Suppose that the hypotheses of either \cref{thm:sm_em_joint} or
\cref{thm:sm_em_relaxed_regularity} hold, and that \(Y\) is continuous.
Then, with \(\hat U^n=n(X-\hat X^n)\),
\begin{equation}\label{eq:sm_continuous_time_error_convergence}
 (Y,\hat U^n)\Rightarrow^{\mathrm{st}}(Y,U)
 \quad\text{in }C([0,1],\mathbb R^{q+d})
\end{equation}
with the uniform topology, stably with respect to \(\mathcal F\).
Here \(\mathcal K^f=0\), so the forcing in
\eqref{eq:sm_em_forcing_limit} is continuous.
\end{corollary}

Its proof is given in
Subsection~\ref{subsec:linear_propagation_main_theorem}.
In particular, for a standard \(q\)-dimensional
\(\mathbb F\)-Brownian motion \(Y=B\), the forcing reduces to
\begin{equation}\label{eq:brownian_forcing_limit_compact}
    \mathcal Z_t^B
    \coloneqq
    \frac12\{f(X_t)-f(x_0)\}
    +\frac1{\sqrt{12}}\int_0^t Df(X_s)\sigma\,dW_s,
\end{equation}
and the limiting error is the unique continuous solution of
\begin{equation}\label{eq:brownian_limit_equation}
    U_t^{B,i}
    =
    \mathcal Z_t^{B,i}+\int_0^t f_j^i(X_s)U_s^{B,j}\,ds.
\end{equation}
These distributional statements do not imply convergence of moments
or strong \(L^p\) error estimates.  For example, the additive equation
\(dX_t=-X_t^3dt+dB_t\), \(X_0=0\), satisfies
\cref{thm:sm_em_joint}, although the absolute moments of its explicit
Euler approximation at time \(1\) diverge as the mesh is refined;
see \cite{hutzenthaler2010strong}.

\paragraph{Jump activity.}
If \(c=0\) and
\[
 \int_0^1\int_{|z|\le1}|z|F_s(dz)\,ds<\infty\quad\text{a.s.},
\]
then (A) holds for every \(C^1\) drift.  This includes finite-activity
drivers and finite-variation drivers with the displayed compensator
integrability.  More generally, if \(0<\gamma\le1\),
\(f\in C^{1,\gamma}_{\mathrm{loc}}\), and
\[
 \int_0^1\int_{|z|\le1}|z|^{1+\gamma}F_s(dz)\,ds<\infty
 \quad\text{a.s.},
\]
then (A) holds.  For a rotationally symmetric \(\alpha\)-stable driver,
\(0<\alpha<2\), the small-jump moment of order \(r\) is finite exactly
when \(r>\alpha\).  Thus \(C^1\) suffices for \(\alpha<1\), while
\(C^{1,\gamma}_{\mathrm{loc}}\) with \(1+\gamma>\alpha\) suffices
for \(1\le\alpha<2\).
For symmetric \(\alpha\)-stable noise with \(1\le\alpha<2\),
a related It\^o formula for
\(C_b^{1,\gamma}\) test functions under \(1+\gamma>\alpha\)
is derived in
\cite[proof of Lemma~4.2, pp.~437--438]{priola2012pathwise}.
These are sufficient conditions, not optimality
claims.  In contrast to condition (B), condition (A) does not require
stationary or independent increments.

\paragraph[Examples with Gaussian and \texorpdfstring{\(\alpha\)}{alpha}-stable noise.]{Examples with Gaussian and \(\alpha\)-stable noise.}
Suppose, as required in condition (B),
that \(Y\) is an \(\mathbb F\)-L\'evy process.
Assume further that the effective noise
\(L=\sigma Y\) has the same
finite-dimensional distributions as
\(\Lambda S^{(\alpha)}+J\), where
\(1<\alpha\le2\), \(\Lambda\) is an
invertible \(d\times d\) matrix,
\(S^{(\alpha)}\) is a rotationally symmetric
\(\alpha\)-stable L\'evy process
(Brownian motion for \(\alpha=2\)),
and \(J\) is an independent arbitrary
\(d\)-dimensional L\'evy process.  Convolution with the law of \(J_t\)
preserves the bounds
\begin{equation}\label{eq:sm_stable_smoothing_bounds}
 \|k_t\|_{L^{p'}}\le Ct^{-d/(\alpha p)},\qquad
 \|\nabla k_t\|_{L^1}\le Ct^{-1/\alpha}.
\end{equation}
A verification is included in
Subsection~\ref{subsec:levy_occupation_estimates}.
Consequently (B) holds whenever
\[
 f\in C^1\cap W^{2,p}_{\mathrm{loc}},\qquad
 p>\max\{1,d/\alpha\}.
\]
This includes nondegenerate Brownian noise, every L\'evy noise with a
nondegenerate Gaussian component, and an invertible linear image of
rotationally symmetric stable noise with any independent L\'evy
component added.  No additional small-jump or large-jump moment
condition is imposed on \(J\).  Neither condition (A) nor condition (B) implies
the other in general.
For a compound Poisson driver, the small-jump
integrability requirement in (A) holds for every
\(C^1\) drift, whereas the density requirement
in (B) fails.
Conversely, the example below satisfies (B)
but not (A).

\paragraph{A multidimensional Sobolev drift.}
Let \(d\ge2\), \(0<\gamma<1\), \(\varepsilon>0\), and choose
\(\chi\in C_c^\infty(\mathbb R^d)\) equal to one near the origin.
Set
\begin{equation}\label{eq:sm_multidimensional_sobolev_example}
 V(x)=\frac12|x|^2+\varepsilon\chi(x)|x|^{2+\gamma},
 \qquad f(x)=-\nabla V(x).
\end{equation}
The drift is globally Lipschitz and \(C^1\), but is not locally
\(C^{1,1}\) at the origin.  Away from the origin it is smooth, and
\(|D^2f(x)|\le C|x|^{\gamma-1}\) near zero.  Hence
\(f\in W^{2,p}_{\mathrm{loc}}\) for
\(1<p<d/(1-\gamma)\).
Since \(d\ge2\), \(0<\gamma<1\), and \(1<\alpha\le2\),
we may take \(p=d\), because
\[
 \max\{1,d/\alpha\}<d<d/(1-\gamma).
\]
Thus condition (B) applies to the L\'evy processes
described above, including multidimensional Brownian
motion and pure-jump stable processes.
To distinguish the two sufficient conditions,
let \(1<\alpha<2\), take \(q=d\) and
\(\sigma=I_d\), and let \(Y\) be a
rotationally symmetric \(\alpha\)-stable
L\'evy process.

Near the origin, where \(\chi=1\), we have
\[
 f(x)=-x-\varepsilon(2+\gamma)|x|^\gamma x.
\]
Consequently, for sufficiently small \(r>0\),
\[
 Df(re_1)-Df(0)
 =-\varepsilon(2+\gamma)r^\gamma
       (I_d+\gamma e_1e_1^\top),
\]
and hence
\[
 \omega_R(r)
 \ge\varepsilon(2+\gamma)(1+\gamma)r^\gamma.
\]
Since the L\'evy measure is
\(F(dz)=c_{d,\alpha}|z|^{-d-\alpha}dz\),
for \(0<\gamma\le\alpha-1\) we obtain
\[
 \int_{|z|\le\varepsilon_\sigma}
       |z|\omega_R(|z|)F(dz)
 \ge C\int_0^\delta r^{\gamma-\alpha}\,dr
 =\infty
\]
for sufficiently small \(\delta>0\).

Thus condition (A) fails for this drift, whereas
condition (B) holds with \(p=d\).
This example separates the two sufficient
conditions without invoking a one-dimensional
occupation-time argument.\looseness=-1

\paragraph[Common uniform marks in a linear equation.]{Common uniform marks in a linear equation.}
Let \(Y\) be the process constructed in
Subsection~\ref{subsec:singular_marked_point_process}. The bound
\(\int_0^1\lambda_tdt<\infty\) and the finite mark set imply that
condition (A) of \cref{thm:sm_em_relaxed_regularity} applies to every
\(C^1\) drift with a nonexplosive solution.

For \(a,\beta>0\), consider
\(dX_t=-\beta X_tdt+a\,dY_t\) with deterministic \(X_0\in\mathbb R^2\).
Write \(\kappa_s\) for the uniform mark attached to the jump at time \(s\).
The same uniform mark is used for all coordinates of a vector jump.
The centered contribution of a jump at \(s\) to the limiting error at
\(t\ge s\) is
\(-a\beta e^{-\beta(t-s)}(\kappa_s-\tfrac12)\Delta Y_s\).
Consequently,
\begin{equation}\label{eq:sm_marked_point_process_covariance}
 \widetilde{\operatorname{Cov}}(U_t^1,U_t^2\mid\mathcal F)
 =\frac{a^2\beta^2}{12}\sum_{0<s\le t}
       e^{-2\beta(t-s)}\Delta Y_s^1\Delta Y_s^2.
\end{equation}
The product in this sum is nonzero precisely for
jumps with mark \((1,1)\). Since \(p_{11}>0\)
and \(\mu_1>0\), the conditional covariance is
strictly positive with positive probability for
every \(t>0\).
If independent uniform marks were used for the
two coordinates, the conditional covariance would
instead vanish, giving an incorrect joint limit.
The formula concerns the limiting error:
for each \(n\), the Euler error is
\(\mathcal F\)-measurable and hence has zero
conditional covariance given \(\mathcal F\).

\section[Proofs of the Euler Error Limits]{Proofs of the Euler Error Limits}
\label{sec:proof_main_theorem}

\subsection[From integrated discretization errors to Euler errors]{From integrated discretization errors to Euler errors}
\label{subsec:integrated_discretization_forcing}

We return to the notation of Section~\ref{sec:setting_main_result}.
Set
\begin{equation}\label{eq:sm_integrated_forcing_n}
    {\mathcal Z}_t^n
    \coloneqq
    n\int_0^{\varphi_n(t)}
    \{f(X_{s-})-f(X_{\varphi_n(s)})\}\,ds.
\end{equation}
Thus \({\mathcal Z}^n=\mathcal I^n(f(X))\).
Introduce
\begin{equation}\label{eq:sm_augmented_process_main_text}
    \Xi_t
    \coloneqq
    \begin{pmatrix}
        Y_t\\
        f(X_t)-f(x_0)
    \end{pmatrix}
    \in\mathbb R^{q+d}.
\end{equation}
A common criterion for all three regularity conditions is stated in
\cref{prop:sm_euler_transfer} below.
Linearity gives
\begin{equation}\label{eq:sm_augmented_companion_blocks}
    \mathcal I_t^n(\Xi)
    =
    \begin{pmatrix}
        \mathcal I_t^n(Y)\\
        {\mathcal Z}_t^n
    \end{pmatrix}.
\end{equation}
On the extension of Section~\ref{sec:setting_main_result}, realize
\eqref{eq:app_generic_first_companion_limit} with \(V=Y\), using the
\(q\)-dimensional Brownian motion \(W\) and the marks fixed there.
In coordinates, this gives
\begin{align}
    \mathcal I_t^\alpha(Y)
    &\coloneqq
    \frac12B_t^{Y,\alpha}
    +
    \frac12Y_t^{c,\alpha}
    +
    \frac1{\sqrt{12}}
    \int_0^t(c_s^{1/2})^{\alpha\gamma}\,dW_s^\gamma
    \notag\\
    &\quad
    +
    \int_0^t\int_{\mathbb R^q}\int_0^1
    uh^\alpha(z)
    (\widetilde\mu^Y-\widetilde\nu^Y)(ds,dz,du)
    \notag\\
    &\quad
    +
    \int_0^t\int_{\mathbb R^q}\int_0^1
    u\{z^\alpha-h^\alpha(z)\}
    \widetilde\mu^Y(ds,dz,du).
    \label{eq:sm_linear_discretization_companion}
\end{align}
In particular, on \(\{T_p\le1\}\),
\begin{equation}\label{eq:sm_driver_companion_jump}
    \Delta\mathcal I(Y)_{T_p}
    =
    \kappa_p\Delta Y_{T_p}.
\end{equation}
When \(Y=B\), \eqref{eq:sm_linear_discretization_companion} gives
\(\mathcal I(B)=B/2+W/\sqrt{12}\), with \(W\) independent of
\(\mathcal F\).

The product truncation is
\begin{equation}\label{eq:app_augmented_truncation}
    \bar h(y,v)
    \coloneqq
    \bigl(h(y),h_f(v)\bigr).
\end{equation}

\begin{proposition}[A criterion for the Euler error limit]
\leavevmode\phantomsection % Keep the PDF target with the proposition after a page break.
\label{prop:sm_euler_transfer}
% Compatibility label for the same Euler error criterion.
\label{prop:sm_augmented_forcing_limit}
Assume \eqref{eq:sm_ac_characteristics} and
\eqref{eq:sm_ac_integrability}, let \(f\in C^1\), and let \(X\) be a
nonexplosive solution of \eqref{eq:sm_em_sde}.  Put
\begin{equation}\label{eq:sm_common_truncated_correction}
 r_s(z)\mathrel{\coloneqq}h_f(\Delta_f(X_{s-},z))-Df(X_{s-})\sigma h(z).
\end{equation}

Suppose that there exists a predictable \(\mathbb R^d\)-valued process \(\beta\) such that
\begin{align}
    f(X_t)-f(x_0)
    &=
    \int_0^t\beta_s\,ds
    +
    \int_0^tDf(X_{s-})\sigma\,dY_s^c
    \notag\\*
    &\quad
    +
    \int_0^t\int_{\mathbb R^q}
    h_f\!\left(\Delta_f(X_{s-},z)\right)
    (\mu^Y-\nu^Y)(ds,dz)
    \notag\\*
    &\quad
    +
    \int_0^t\int_{\mathbb R^q}
    \left\{
        \Delta_f(X_{s-},z)
        -
        h_f\!\left(\Delta_f(X_{s-},z)\right)
    \right\}\mu^Y(ds,dz)\mathpunct{.}
    \label{eq:app_augmented_fX_decomposition}
\end{align}
Assume further that
\begin{equation}\label{eq:sm_common_calculus_integrability}
 \int_0^1|\beta_s|ds
 +\int_0^1\int_{\mathbb R^q}|r_s(z)|F_s(dz)ds
 +\sum_{0<s\le1}|\Theta_f(X_{s-},\Delta Y_s)|<\infty
 \qquad\text{a.s.}
\end{equation}

Then \eqref{eq:sm_em_joint_convergence} holds with \(U\) defined by
\eqref{eq:sm_em_limit_equation}, in the vector-valued \(J_1\)-topology
and stably with respect to \(\mathcal F\).
\end{proposition}

\begin{proof}

\medskip
\noindent
\emph{Step 1: Characteristics of the augmented process.}\nopagebreak[4]

The canonical representation of \(f(X)\) relative to \(h_f\) is specified by
\eqref{eq:app_augmented_fX_decomposition}.
Let \(\Xi\) be defined by
\eqref{eq:sm_augmented_process_main_text}.  Then \(\Xi\) is a
\((q+d)\)-dimensional c\`adl\`ag semimartingale.  Its
predictable characteristics relative to \(\bar h\) are absolutely continuous:
\begin{equation}\label{eq:app_augmented_characteristics}
    B_t^\Xi
    =
    \int_0^t\bar b_s\,ds,
    \qquad
    C_t^\Xi
    =
    \int_0^t\bar c_s\,ds,
    \qquad
    \nu^\Xi(ds,dw)
    =
    \bar F_s(dw)\,ds.
\end{equation}
They are given as follows.  Set
\begin{equation}\label{eq:app_augmented_A_matrix}
    A_s
    =
    \begin{pmatrix}
        I_q\\
        Df(X_{s-})\sigma
    \end{pmatrix}.
\end{equation}
Then
\begin{equation}\label{eq:app_augmented_covariance}
    \bar c_s
    =
    A_sc_sA_s^\top.
\end{equation}
For every nonnegative Borel function \(g\) on \(\mathbb R^{q+d}\),
\begin{equation}\label{eq:app_augmented_jump_kernel}
    \int_{\mathbb R^{q+d}}g(w)\bar F_s(dw)
    =
    \int_{\mathbb R^q}
    g\!\left(
        z,
        \Delta_f(X_{s-},z)
    \right)F_s(dz).
\end{equation}
Finally,
\begin{equation}\label{eq:app_augmented_drift}
    \bar b_s
    =
    \begin{pmatrix}
        b_s\\
        \beta_s
    \end{pmatrix},
\end{equation}
{%
where \(\beta\) is a predictable drift density, unique
\(ds\otimes d\mathbb P\)-almost everywhere.
}
Moreover, \(\Xi\) satisfies the integrability condition required in
\cref{thm:integrated_discretization_limit}:
\begin{equation}\label{eq:app_augmented_integrability}
    \int_0^1
    \left(
        |\bar b_s|
        +
        \operatorname{tr}(\bar c_s)
        +
        \int_{\mathbb R^{q+d}}(1\wedge|w|^2)\bar F_s(dw)
    \right)ds
    <\infty
    \qquad\text{a.s.}
\end{equation}

Since \(X_-\) is predictable and \(F_s(dz)\) is a predictable kernel,
the map \((\omega,s,z)\mapsto(z,\Delta_f(X_{s-},z))\) is predictable
in \((\omega,s)\) and Borel in \(z\).  Thus its pushforward kernel
\(\bar F_s\) in \eqref{eq:app_augmented_jump_kernel} is predictable.
The jump measure of \(\Xi\) is the image of \(\mu^Y\) under this map.
For every nonnegative predictable test function \(g\), compensation gives
\[
    \mathbb E[(g*\mu^\Xi)_t]
    =\mathbb E\int_0^t\int_{\mathbb R^q}
      g\bigl(s,(z,\Delta_f(X_{s-},z))\bigr)F_s(dz)\,ds.
\]
Hence \(\bar\nu(ds,dw)=\bar F_s(dw)\,ds\) is the compensator of \(\mu^\Xi\).
The continuous local martingale
part is \(\int_0^\cdot A_s\,dY_s^c\); this identifies
\(\bar c=A c A^\top\), including the cross-covariances.
The product truncation \(\bar h\) from
\eqref{eq:app_augmented_truncation} is bounded, Borel and equal to the
identity near zero.  Combining the canonical representation of \(Y\) relative to \(h\) with
\eqref{eq:app_augmented_fX_decomposition} gives the canonical representation
of \(\Xi\) relative to \(\bar h\). Thus \(\Xi\) has the stated characteristics, with drift density \(\bar b\).

Fix a deterministic \(R>|x_0|\) and set
\[
    \tau_R\coloneqq\inf\{t\in[0,1]:|X_t|>R\}\wedge1,
    \qquad \inf\varnothing=\infty.
\]
Then \(|X_{s-}|\le R\) for \(0<s\le\tau_R\).
On \(s\le\tau_R\), local boundedness of \(Df\) gives
\[
 1\wedge|(z,\Delta_f(X_{s-},z))|^2\le C_R(1\wedge|z|^2),
 \qquad \operatorname{tr}(A_sc_sA_s^\top)\le C_R\operatorname{tr}(c_s).
\]
Together with the integrability of \(\beta\), these estimates prove
\eqref{eq:app_augmented_integrability} by nonexplosion.

\medskip
\noindent
\emph{Step 2: Limit of the integrated discretization error.}\nopagebreak[4]

We next prove
\begin{equation}\label{eq:app_augmented_forcing_convergence}
    (Y^{[n]},{\mathcal Z}^n)
    \Rightarrow^{\mathrm{st}}
    (Y,\mathcal Z)
\end{equation}
in \(D([0,1],\mathbb R^{q+d})\) with the vector-valued \(J_1\)-topology,
stably with respect to \(\mathcal F\).

{%
{%
By Step 1,
\cref{thm:integrated_discretization_limit} applies to \(\Xi\)
with truncation \(\bar h\). Hence
}
}
\begin{equation}\label{eq:app_full_augmented_companion_convergence}
    \left(
        \Xi^{[n]},
        \mathcal I^n(\Xi)
    \right)
    \Rightarrow^{\mathrm{st}}
    \left(
        \Xi,
        \mathcal I(\Xi)
    \right)
\end{equation}
in \(D([0,1],\mathbb R^{2(q+d)})\), stably with respect to \(\mathcal F\).
The deterministic projection onto the first \(q\) coordinates of
\(\Xi^{[n]}\) and the last \(d\) of \(\mathcal I^n(\Xi)\) is continuous
for the vector-valued \(J_1\)-topology.  By the continuous mapping theorem,
\((Y^{[n]},{\mathcal Z}^n)\) converges stably to the pair consisting of
\(Y\) and the last \(d\) coordinates of \(\mathcal I(\Xi)\).

Stable limits are determined by their \(\mathcal F\)-conditional law.
On the already fixed very good extension, use the existing
\(q\)-dimensional Brownian motion \(W\) to define
\[
    G_t\coloneqq\int_0^t A_sc_s^{1/2}\,dW_s.
\]
The \((q+d)\times q\) matrix \(A_sc_s^{1/2}\) satisfies
\[
    (A_sc_s^{1/2})(A_sc_s^{1/2})^\top
    =A_sc_sA_s^\top=\bar c_s.
\]
Conditionally on \(\mathcal F\), \(G\) is a centered continuous Gaussian
process with independent increments and covariance density
\(A_sc_sA_s^\top=\bar c_s\). Therefore, \(G/\sqrt{12}\) has the same
conditional law as the auxiliary Gaussian term in \(\mathcal I(\Xi)\)
supplied by \cref{thm:integrated_discretization_limit}.
The jump vector
\(\Delta\Xi_s=(\Delta Y_s,\Delta_f(X_{s-},\Delta Y_s))\) is nonzero
exactly when \(\Delta Y_s\ne0\), so \(\Xi\) and \(Y\) have the same
jump times and we may use the same exhausting sequence \((T_p)\)
and uniform marks \((\kappa_p)\).
The marked jump measure is the image of
\(\widetilde\mu^Y\) under
\[
    (s,z,u)\longmapsto
    \bigl(s,(z,\Delta_f(X_{s-},z)),u\bigr).
\]
At each jump time,
\[
    \Delta\mathcal I(\Xi)_{T_p}
    =\kappa_p\Delta\Xi_{T_p}
    =\begin{pmatrix}
        \kappa_p\Delta Y_{T_p}\\
        \kappa_p\Delta_f(X_{T_p-},\Delta Y_{T_p})
      \end{pmatrix}.
\]
Conditionally on \(\mathcal F\), the marks \((\kappa_p)\) are i.i.d.
uniform on \([0,1]\). Applying these marks to the jumps of \(\Xi\)
therefore gives the required conditional law of the jump terms in
\(\mathcal I(\Xi)\). Since the marks are also independent of \(W\),
these jump terms and \(G\) are conditionally independent given \(\mathcal F\).
Thus the conditional law of \(\mathcal I(\Xi)\) can be realized using
the same \(W\) and the same mark at each original jump.  By \eqref{eq:app_generic_first_companion_compact}, the last
\(d\) coordinates of \(\mathcal I(\Xi)\) are precisely \(\mathcal Z\), which
identifies the projected limit and proves
\eqref{eq:app_augmented_forcing_convergence}.

By \eqref{eq:sm_common_calculus_integrability},
\((|r|*\nu^Y)_1<\infty\) almost surely. Localization and the
compensation formula then imply that \((|r|*\mu^Y)_1<\infty\)
almost surely. Hence the compensated integral can be written
componentwise as \(r*(\mu^Y-\nu^Y)=r*\mu^Y-r*\nu^Y\).
Since \(r_s(z)\) does not depend on the mark \(u\), the same
integrability holds for \(|r|*\widetilde\mu^Y\) and
\(|r|*\widetilde\nu^Y\).

{%
To prove independence from \(h_f\), let \({\mathcal K}(Y)\) be the
}
process defined in \eqref{eq:app_centered_marked_companion} with
\(V=Y\).  Thus
\(\mathcal I(Y)=Y/2+\int_0^\cdot c_s^{1/2}\,dW_s/\sqrt{12}
+{\mathcal K}(Y)\).
Using the correction \(r\) from
\eqref{eq:sm_common_truncated_correction} and its integrability in
\eqref{eq:sm_common_calculus_integrability} gives
\begin{equation}\label{eq:app_centered_remainder_identification}
    \mathcal K_t^f
    =\int_0^t Df(X_{s-})\sigma\,d{\mathcal K}_s(Y)
     +\int_0^t\int_{\mathbb R^q}\int_0^1
       \left(u-\frac12\right)\Theta_f(X_{s-},z)
       \widetilde\mu^Y(ds,dz,du).
\end{equation}
Indeed, expanding the integral against \({\mathcal K}(Y)\) gives
\[
\begin{aligned}
    \int_0^\cdot Df(X_{s-})\sigma\,d{\mathcal K}_s(Y)
    &=\{(u-\tfrac12)Df(X_{s-})\sigma h(z)\}
       *(\widetilde\mu^Y-\widetilde\nu^Y)\\
    &\quad+\{(u-\tfrac12)Df(X_{s-})\sigma(z-h(z))\}
       *\widetilde\mu^Y.
\end{aligned}
\]
Subtracting this from the definition of \(\mathcal K^f\) yields
\[
\begin{aligned}
    &\mathcal K^f-\int_0^\cdot Df(X_{s-})\sigma\,d{\mathcal K}_s(Y)\\
    &\qquad=\{(u-\tfrac12)r_s(z)\}*(\widetilde\mu^Y-\widetilde\nu^Y)\\
    &\qquad\quad+\{(u-\tfrac12)(\Theta_f(X_{s-},z)-r_s(z))\}
        *\widetilde\mu^Y.
\end{aligned}
\]
The local integrability established above for
\(|r|*\widetilde\mu^Y\) and \(|r|*\widetilde\nu^Y\), together with
\(|u-\tfrac12|\le1/2\), allows the compensated integral to be written
componentwise as
\[
    \{(u-\tfrac12)r_s(z)\}*(\widetilde\mu^Y-\widetilde\nu^Y)
    =\{(u-\tfrac12)r_s(z)\}*\widetilde\mu^Y
     -\{(u-\tfrac12)r_s(z)\}*\widetilde\nu^Y.
\]
The difference between the two sides of
\eqref{eq:app_centered_remainder_identification} is therefore
\(-\{(u-\tfrac12)r_s(z)\}*\widetilde\nu^Y=0\), by Fubini and
\(\int_0^1(u-\tfrac12)\,du=0\); the \(\Theta_f\)-integral is
absolutely convergent against \(\widetilde\mu^Y\) by
\eqref{eq:sm_common_calculus_integrability}.
{%
{%
The right-hand side of \eqref{eq:app_centered_remainder_identification}
contains no \(h_f\); hence changing \(h_f\) leaves \(\mathcal K^f\),
and therefore \(\mathcal Z\), unchanged up to indistinguishability
when the same \(W\) and jump marks are used.
}
}

\medskip
\noindent
\emph{Step 3: Joint convergence with the averaged Jacobian.}\nopagebreak[4]

Define the averaged Jacobian
\begin{equation}\label{eq:sm_new_mean_value_coefficient}
    H_s^n
    \coloneqq
    \int_0^1
    Df\!\left(
        \hat X_{\varphi_n(s)}^n
        +
        r\{X_{\varphi_n(s)}-\hat X_{\varphi_n(s)}^n\}
    \right)dr.
\end{equation}

Set \(H=Df(X_-)\).
Step 2 gives the stable convergence of
\((Y^{[n]},{\mathcal Z}^n)\), while \cref{le:sm_mean_value_coefficient_convergence}
gives \(H^n\to H\) in probability in \(L^1\).
By stable Slutsky's theorem,
\begin{equation}\label{eq:app_forcing_coefficient_joint_limit}
    (Y^{[n]},{\mathcal Z}^n,H^n)
    \Rightarrow^{\mathrm{st}}
    (Y,\mathcal Z,H)
\end{equation}
in \(D([0,1],\mathbb R^{q+d})
    \times
    L^1([0,1],\mathbb R^{d\times d}),\)
where the first factor carries the usual vector-valued \(J_1\)-topology.
Moreover,
\[
    \|H^n\|_{L^\infty([0,1])}=O_{\mathbb P}(1).
\]
The stable convergence is with respect to the original sigma-field.

\medskip
\noindent
\emph{Step 4: Transfer to the Euler error.}\nopagebreak[4]

In \eqref{eq:sm_em_error_identity}, add and subtract
\(f(X_{\varphi_n(s)})\).  The mean-value formula then gives the exact identity
\begin{equation}\label{eq:sm_exact_grid_error_equation}
    \bar U_t^n
    =
    {\mathcal Z}_t^n
    +
    \int_0^{\varphi_n(t)}
    H_s^n\bar U_s^n\,ds,
    \qquad 0\le t\le1.
\end{equation}

We now use the joint convergence and the boundedness in probability of \(\|H^n\|_{L^\infty([0,1])}\) established in Step 3.
Apply \cref{le:grid_linear_propagation} with \(r=q\),
\(\Gamma^n=Y^{[n]}\), \(J^n=\mathcal Z^n\), \(U^n=\bar U^n\),
\(H^n\) as in \eqref{eq:sm_new_mean_value_coefficient}, and
\((\Gamma,J,H)=(Y,\mathcal Z,Df(X_-))\). The limiting equation is precisely
\eqref{eq:sm_em_limit_equation}, whose unique solution is \(U\).
This proves \eqref{eq:sm_em_joint_convergence}.
\end{proof}

\subsection{Proofs of the main results}
\label{subsec:verification_regularity_conditions}
\label{subsec:linear_propagation_main_theorem}
\label{subsec:proof_relaxed_regularity}

We verify the hypotheses of \cref{prop:sm_euler_transfer} under the
assumptions of \cref{thm:sm_em_joint} and under each of the two alternatives
in \cref{thm:sm_em_relaxed_regularity}. We then prove
\cref{cor:sm_continuous_time_error}.

\begin{proof}[Proof of \cref{thm:sm_em_joint}]
By \cref{prop:sm_euler_transfer}, it suffices to establish
\eqref{eq:app_augmented_fX_decomposition} with a predictable
process \(\beta\) and to verify
\eqref{eq:sm_common_calculus_integrability}.

Fix a deterministic \(R>|x_0|\) and set
\[
    \tau_R\coloneqq\inf\{t\in[0,1]:|X_t|>R\}\wedge1,
    \qquad \inf\varnothing=\infty.
\]
Then \(|X_{s-}|\le R\) for \(0<s\le\tau_R\).
For \(|x|\le R\) and \(|z|\le(1+|\sigma|)^{-1}\),
\[
    |x+\theta\sigma z|\le |x|+|\sigma||z|\le R+1,
    \qquad \theta\in[0,1],
\]
so \(x+\theta\sigma z\in\overline B(0,R+1)\).
{%
{%
We have
}
\[
 \Theta_f(x,z)=\int_0^1
       \{Df(x+u\sigma z)-Df(x)\}\sigma z\,du.
\]
Boundedness and Lipschitz continuity of \(Df\) on this ball therefore
give a deterministic constant \(K_R<\infty\) such that
}
\begin{equation}\label{eq:app_augmented_small_jump_Taylor}
    |\Delta_f(x,z)|
    \le
    K_R|z|,
    \qquad
    |\Theta_f(x,z)|
    \le
    K_R|z|^2,
    \qquad
    |z|\le(1+|\sigma|)^{-1}.
\end{equation}

With \(K_R\) fixed, choose \(0<\varepsilon_R\le(1+|\sigma|)^{-1}\) such that \(h(z)=z\)
for \(|z|\le\varepsilon_R\) and \(h_f(v)=v\) for
\(|v|\le K_R\varepsilon_R\).
Below, \(C_R<\infty\) denotes a deterministic constant, possibly changing
from line to line, which may depend on \(R\) and the fixed data.
Recall the definition of \(r\) in \eqref{eq:sm_common_truncated_correction}.
For \(s\le\tau_R\) and \(|z|\le\varepsilon_R\), we have
\(r_s(z)=\Theta_f(X_{s-},z)\).  For \(|z|>\varepsilon_R\), use boundedness
of the truncations; together with the Taylor bounds, this gives, for all
\(z\) and \(|x|\le R\),
\begin{align}
    1\wedge
    \left|
        \begin{pmatrix}
            z\\
            \Delta_f(x,z)
        \end{pmatrix}
    \right|^2
    &\le
    C_R(1\wedge|z|^2),
    \label{eq:app_augmented_jump_integrability_bound}\\
    \left|
        h_f(\Delta_f(x,z))-Df(x)\sigma h(z)
    \right|
    &\le
    C_R(1\wedge|z|^2).
    \label{eq:app_augmented_drift_correction_bound}
\end{align}
Since \(X_-\) and \(\mathbf1_{\{s\le\tau_R\}}\) are predictable,
so is the function
\((s,z)\mapsto |r_s(z)|\mathbf1_{\{s\le\tau_R\}}\).
By \eqref{eq:app_augmented_drift_correction_bound} and
\eqref{eq:sm_ac_integrability}, its integral against \(\nu^Y\) is
finite almost surely. Localization and the compensation formula give
the same conclusion for its integral against \(\mu^Y\). Consequently,
the compensated integral can be written componentwise as
\[
    (r\mathbf1_{\{s\le\tau_R\}})*(\mu^Y-\nu^Y)
    =(r\mathbf1_{\{s\le\tau_R\}})*\mu^Y
     -(r\mathbf1_{\{s\le\tau_R\}})*\nu^Y.
\]

{%
The quadratic estimate and square-summability of semimartingale jumps,
together with the fact that only finitely many jumps have magnitude greater
than \(\varepsilon_R\), give
\[
 \sum_{0<s\le\tau_R}|\Theta_f(X_{s-},\Delta Y_s)|<\infty
 \quad\text{a.s.}
\]
Large-jump terms are a finite pathwise sum, including a possible exit
jump.  As \(\tau_R=1\) eventually along each nonexplosive path, the
sum is well defined on \([0,1]\).

Let \(\rho\) be a smooth nonnegative mollifier supported in the unit ball,
with \(\int\rho=1\), and set
\(\rho_m(x)\mathrel{{\coloneqq}}m^d\rho(mx)\), \(f_m\mathrel{{\coloneqq}}f*\rho_m\).
Locally uniformly, \(f_m\to f\) and \(Df_m\to Df\).
{%
Since \(Df\) is Lipschitz on \(\overline B(0,R+2)\),
\(Df_m\) is Lipschitz on \(\overline B(0,R+1)\) with a constant
independent of \(m\ge1\). Hence
\begin{equation}\label{eq:app_mollified_hessian_bound}
 \sup_{m\ge1}\sup_{|x|\le R+1}
       \|D^2 f_m^i(x)\|_{\mathrm{op}}\le C_R,
 \qquad 1\le i\le d.
\end{equation}
}
Keep \(X\) fixed.  Applying It\^o's formula to \(f_m(X)\) gives
\[
 f_m(X_t)-f_m(x_0)=\int_0^t Df_m(X_{s-})\,dX_s+R_t^m
                  +\sum_{0<s\le t}\Theta_{f_m}(X_{s-},\Delta Y_s),
\]
where \(\Theta_{f_m}\) is defined as in \eqref{eq:sm_theta_def_section},
and, with \(Q_s\mathrel{{\coloneqq}}\sigma c_s\sigma^\top\),
\[
 R_t^{m,i}\mathrel{{\coloneqq}}\frac12\int_0^t\partial_{jk}f_m^i(X_{s-})
                \sigma_\alpha^j c_s^{\alpha\beta}\sigma_\beta^k\,ds.
\]
{%
Define \(R^f\) by
}
\begin{equation}\label{eq:app_augmented_residual_definition}
 R_t^f\mathrel{{\coloneqq}}f(X_t)-f(x_0)-\int_0^t Df(X_{s-})\,dX_s
          -\sum_{0<s\le t}\Theta_f(X_{s-},\Delta Y_s).
\end{equation}
{%
The stopped path \(X^{\tau_R}\) has bounded range almost surely,
including its value at exit, whose norm may exceed \(R\).
Local uniform convergence of \(f_m\) therefore gives
\[
 \sup_{t\le1}
 |f_m(X_{t\wedge\tau_R})-f(X_{t\wedge\tau_R})|
 \longrightarrow0\quad\text{a.s.}
\]
Since \(|X_{s-}|\le R\) for \(0<s\le\tau_R\), the predictable
integrands \(\mathbf1_{\{s\le\tau_R\}}(Df_m-Df)(X_{s-})\)
are uniformly bounded and converge pointwise. Thus
the dominated convergence theorem for stochastic integrals gives
\[
 \int_0^{\cdot\wedge\tau_R}(Df_m-Df)(X_{s-})\,dX_s
 \longrightarrow0\quad\text{ucp}.
\]
The uniform quadratic small-jump bound and the finite large-jump sum
give, by dominated convergence,
\[
 \sum_{0<s\le\tau_R}
 |(\Theta_{f_m}-\Theta_f)(X_{s-},\Delta Y_s)|
 \longrightarrow0\quad\text{a.s.}
\]
The total absolute sum also bounds the uniform difference of the
stopped jump sums.
}  Subtracting these terms from the smooth It\^o
identities proves \((R^m)^{\tau_R}\to(R^f)^{\tau_R}\) in ucp.

For every \(0\le u\le v\le1\), the uniform Hessian bound implies
\[
 |R^m_{v\wedge\tau_R}-R^m_{u\wedge\tau_R}|
 \le C_R\int_{u\wedge\tau_R}^{v\wedge\tau_R}\operatorname{tr}(Q_s)\,ds.
\]
An almost surely uniformly convergent subsequence gives the same bound
for the continuous limit \((R^f)^{\tau_R}\).  Sum over arbitrary finite
partitions of each interval and take the supremum.  {%
Almost surely, the following bound holds simultaneously for all
\(0\le u\le v\le1\):
\[
 \operatorname{Var}((R^f)^{\tau_R})_v
 -\operatorname{Var}((R^f)^{\tau_R})_u
 \le C_R\int_u^v
 \mathbf1_{\{s\le\tau_R\}}\operatorname{tr}(Q_s)\,ds.
\]
Equivalently, as measures on \([0,1]\),
}
\begin{equation}\label{eq:app_augmented_residual_variation_bound}
 d\operatorname{Var}((R^f)^{\tau_R})_s
 \le C_R\mathbf1_{\{s\le\tau_R\}}
                  \operatorname{tr}(Q_s)\,ds.
\end{equation}
Since \(\operatorname{tr}(Q_s)\le\|\sigma\|_{\mathrm{op}}^2
\operatorname{tr}(c_s)\), the right-hand side is integrable.
Taking a countable increasing sequence of radii shows that \(R^f\)
is continuous and absolutely continuous in time on \([0,1]\).
{%
Substituting \eqref{eq:sm_em_sde} into
\eqref{eq:app_augmented_residual_definition} gives
\begin{equation}\label{eq:app_augmented_generalized_ito}
\begin{split}
 f(X_t)-f(x_0)
 &=\int_0^t Df(X_{s-})f(X_{s-})\,ds
   +\int_0^t Df(X_{s-})\sigma\,dY_s\\
 &\quad+R_t^f
   +\int_0^t\int_{\mathbb R^q}
      \Theta_f(X_{s-},z)\,\mu^Y(ds,dz).
\end{split}
\end{equation}
By \eqref{eq:app_augmented_generalized_ito}, \(f(X)\) is a semimartingale,
with continuous local martingale
part \(\int_0^\cdot Df(X_{s-})\sigma\,dY_s^c\).
}

We now express the decomposition of \(f(X)\) relative to the truncation \(h_f\).
The stopped identity for \(r\) gives
\(-r*(\mu^Y-\nu^Y)+r*\mu^Y=+r*\nu^Y\), keeping the first-order
small-jump terms compensated.  The drift process is
\begin{equation}\label{eq:app_augmented_drift_process}
\begin{split}
 B_t^{f(X)}
 &=\int_0^t Df(X_{s-})f(X_{s-})\,ds
      +\int_0^t Df(X_{s-})\sigma b_s\,ds\\
 &\quad+R_t^f+\int_0^t\int_{\mathbb R^q}r_s(z)F_s(dz)\,ds.
\end{split}
\end{equation}
The process \(B^{f(X)}\) is continuous and adapted, hence predictable.
Its paths are absolutely continuous, so
\cite[Proposition~I.3.13]{jacod2003limit}, applied componentwise, yields
a predictable process \(\beta\) such that
\[
 B_t^{f(X)}=\int_0^t\beta_s\,ds.
\]
The density is unique \(ds\otimes d\mathbb P\)-almost everywhere.
}
This gives \eqref{eq:app_augmented_fX_decomposition}.
The identification of the jump kernel and continuous covariance from
this decomposition is the same as in Step 1 of the proof of
\cref{prop:sm_euler_transfer}.

For the fixed \(R\) and \(\tau_R\) above, there is a deterministic
constant \(C_R<\infty\) such that, \(ds\otimes d\mathbb P\)-almost
everywhere on \(s\le\tau_R\),
\begin{equation}\label{eq:app_augmented_drift_bound}
\begin{split}
    &|\bar b_s|+\operatorname{tr}(\bar c_s)
    +\int(1\wedge|w|^2)\bar F_s(dw)\\
    &\qquad\le C_R\left\{1+|b_s|+\operatorname{tr}(c_s)
                  +\int(1\wedge|z|^2)F_s(dz)\right\}.
\end{split}
\end{equation}

{%
The drift process and \eqref{eq:app_augmented_residual_variation_bound}
bound \(|\beta_s|\) by the right-hand side of
\eqref{eq:app_augmented_drift_bound}, \(ds\otimes d\mathbb P\)-almost
everywhere on \(s\le\tau_R\).  Also
\[
 \operatorname{tr}(A_sc_sA_s^\top)
 \le\|A_s\|_{\mathrm{op}}^2\operatorname{tr}(c_s).
\]
Together with \eqref{eq:app_augmented_jump_integrability_bound},
these estimates prove \eqref{eq:app_augmented_drift_bound}.
}
The right-hand side is integrable up to \(\tau_R\) by \eqref{eq:sm_ac_integrability}.

For \(s\le\tau_R\), the same estimates justify the limiting integrals in
\eqref{eq:sm_centered_marked_forcing} and
\eqref{eq:sm_em_forcing_limit}:
\begin{align}
    |h_f(\Delta_f(X_{s-},z))|^2
    &\le C_R(1\wedge|z|^2),
    \label{eq:sm_compact_forcing_small_jump_bound}\\
    \Delta_f(X_{s-},z)-h_f(\Delta_f(X_{s-},z))&=0
    \quad (|z|\le\varepsilon_R),
    \label{eq:sm_compact_forcing_large_jump_support}\\
    |Df(X_{s-})\sigma c_s^{1/2}|^2&\le C_R\operatorname{tr}(c_s).
    \notag
\end{align}
For the centered uniform mark \(u-\tfrac12\),
\[
    \int_0^1\left(u-\frac12\right)^2du=\frac1{12}.
\]
Thus the compensated marked integral in
\eqref{eq:sm_centered_marked_forcing} has locally finite predictable
bracket.  The uncompensated integral containing
\(\Delta_f(X_{s-},z)-h_f(\Delta_f(X_{s-},z))\) is supported on
\(|z|>\varepsilon_R\), so it has only finitely many jumps after stopping
at \(\tau_R\).  The same estimates, with \(h(z)\) and \(c^{1/2}\), show that
\(\mathcal I(Y)\) is well defined.
{%
The absolute convergence of the \(\Theta_f\)-sum was proved directly
above using the local quadratic bound and the finite large-jump sum.
}
Projection of the marked measure and compensator gives
\(|r|*\widetilde\mu^Y=|r|*\mu^Y\) and
\(|r|*\widetilde\nu^Y=|r|*\nu^Y\), so both stopped processes are locally
integrable by the argument above.  The marked jump remainder is also
absolutely convergent because the marks are in \([0,1]\).
Since \(X\) is nonexplosive and c\`adl\`ag on \([0,1]\),
\[
    \sup_{0\le s\le1}|X_s|<\infty\quad\text{almost surely}.
\]
Hence \(\tau_R=1\) for all sufficiently large \(R\), almost surely.
Letting \(R\uparrow\infty\) removes the localization, proves
\eqref{eq:app_augmented_integrability}, and justifies the limiting
stochastic integrals on \([0,1]\).
The decomposition \eqref{eq:app_augmented_fX_decomposition},
the drift bound \eqref{eq:app_augmented_drift_bound}, and the estimates
for \(r\) and \(\Theta_f\) above give
\eqref{eq:sm_common_calculus_integrability}.

Applying \cref{prop:sm_euler_transfer} proves
\eqref{eq:sm_em_joint_convergence}.
\end{proof}

\begin{proof}[Proof of \cref{thm:sm_em_relaxed_regularity}]
It suffices to verify the hypotheses of
\cref{prop:sm_euler_transfer}.
This criterion does not require the deterministic pointwise bound
\eqref{eq:app_augmented_drift_bound} under the relaxed assumptions.

\emph{Condition (A).}
Fix \(R>|x_0|\) and let
\(\tau_R=\inf\{t\in[0,1]:|X_t|>R\}\wedge1\), as in the proof of
\cref{thm:sm_em_joint}.
For \(|x|\le R\) and \(|z|\le\varepsilon_\sigma\),
\begin{equation}\label{eq:sm_relaxed_jump_remainder_bound}
 |\Theta_f(x,z)|
 \le \|\sigma\|_{\mathrm{op}}|z|
           \omega_R(\|\sigma\|_{\mathrm{op}}|z|).
\end{equation}
This follows by integrating
\(\{Df(x+u\sigma z)-Df(x)\}\sigma z\) over \(u\in[0,1]\).
The right-hand side is integrable against
\(\mathbf1_{\{s\le\tau_R\}}F_s(dz)ds\) on the indicated small-jump
region by \eqref{eq:sm_jump_modulus_condition}.
Its compensator integral is continuous and finite-valued.  Localizing
that integral and applying compensation shows that the corresponding
sum over actual small jumps is finite almost surely.  Jumps exceeding
\(\varepsilon_\sigma\) form a finite pathwise sum.

Let \(\varrho\) be a smooth nonnegative mollifier supported in the unit
ball and put \(f_m=f*\varrho_m\).  On the same region,
\eqref{eq:sm_relaxed_jump_remainder_bound} holds for \(f_m\), uniformly
in \(m\ge1\): the shifted points used in the convolution lie in
\(\overline B(0,R+2)\), which is the ball used in
\eqref{eq:sm_drift_modulus}.  Also \(f_m\to f\) and \(Df_m\to Df\)
locally uniformly.
Because \(c=0\), the continuous quadratic-variation term in
It\^o's formula for \(f_m(X)\) vanishes.
The stochastic-integral convergence and the treatment of a finite exit
jump are exactly as in the mollification argument in the proof of
\cref{thm:sm_em_joint}.
The small-jump sums converge in total variation by the preceding
summable bound.  We obtain, first stopped at \(\tau_R\) and then by
nonexplosion on the full interval,
\begin{equation}\label{eq:sm_groupA_generalized_ito}
 f(X_t)-f(x_0)
 =\int_0^t Df(X_{s-})\,dX_s
        +\sum_{0<s\le t}\Theta_f(X_{s-},\Delta Y_s).
\end{equation}

Recall the definition of \(r\) in \eqref{eq:sm_common_truncated_correction}.
On \(s\le\tau_R\), choose a deterministic
\(0<\varepsilon_R\le\varepsilon_\sigma\) small enough that
\(r_s(z)=\Theta_f(X_{s-},z)\) when \(|z|\le\varepsilon_R\).
For larger \(z\), boundedness of the two truncations and local
boundedness of \(Df\) give \(|r_s(z)|\le C_R\), while
\[
 \int_0^1 F_s(\{|z|>\varepsilon_R\})ds<\infty\quad\text{a.s.}
\]
Thus \(|r|*\nu^Y\) is finite on each stopped interval and hence on
\([0,1]\).  Using the same change of truncation as in
\eqref{eq:app_augmented_drift_process}, formula
\eqref{eq:sm_groupA_generalized_ito} gives
\eqref{eq:app_augmented_fX_decomposition} with
\begin{equation}\label{eq:sm_groupA_beta}
 \beta_s=Df(X_{s-})\{f(X_{s-})+\sigma b_s\}
                    +\int_{\mathbb R^q}r_s(z)F_s(dz).
\end{equation}
Its absolute time integral is finite by the original characteristic
assumptions, the preceding estimates, and nonexplosion.
The absolute summability of the Taylor remainders
\(\Theta_f(X_{s-},\Delta Y_s)\) over \(0<s\le1\) was already established, proving
\eqref{eq:sm_common_calculus_integrability}.

\emph{Condition (B).}
\Cref{le:levy_smoothing_local_characteristics} in
Appendix~\ref{app:levy_smoothing} verifies precisely the decomposition
and integrability required by \cref{prop:sm_euler_transfer}.
Applying \cref{prop:sm_euler_transfer} completes the proof in this case
as well.
\end{proof}

\begin{proof}[Proof of \cref{cor:sm_continuous_time_error}]
Since \(Y\) is continuous, \(X\) is continuous and \(\mathcal K^f=0\).
The forcing in \eqref{eq:sm_em_forcing_limit} and the solution of
\eqref{eq:sm_em_limit_equation} are therefore continuous.
The applicable result, \cref{thm:sm_em_joint} or
\cref{thm:sm_em_relaxed_regularity}, gives
\begin{equation}\label{eq:brownian_grid_time_limit}
    (Y^{[n]},\bar U^n)
    \Rightarrow^{\mathrm{st}}
    (Y,U)
\end{equation}
in the vector-valued \(J_1\)-topology, stably with respect to
\(\mathcal F\).  In particular, \(\|\bar U^n\|_\infty=O_{\mathbb P}(1)\).
For a continuous path \(g\), let
\(\omega_g(\delta)\coloneqq\sup_{|s-t|\le\delta}|g_s-g_t|\).
Then \(\|Y-Y^{[n]}\|_\infty\le\omega_Y(1/n)\to0\) almost surely.

Put \(\Omega_{\rho,n}\coloneqq\{\sup_{t\le1}(|X_t|\vee|\hat X_t^n|)\le\rho\}\)
and let \(L_\rho\) be a Lipschitz constant of \(f\) on this ball.
By \eqref{eq:sm_em_error_identity},
\[
    \hat U_t^n-\bar U_t^n
    =n\int_{\varphi_n(t)}^t
       \{f(X_s)-f(\hat X_{\varphi_n(t)}^n)\}\,ds.
\]
Using \(X_{\varphi_n(t)}-\hat X_{\varphi_n(t)}^n
=\bar U_{\varphi_n(t)}^n/n\), we obtain
\begin{equation}\label{eq:brownian_mesh_error_bound}
    \|\hat U^n-\bar U^n\|_\infty
    \le L_\rho\omega_X(1/n)
       +\frac{L_\rho}{n}\|\bar U^n\|_\infty
    \qquad\text{on }\Omega_{\rho,n}.
\end{equation}
The right-hand side tends to zero in probability for fixed \(\rho\).
Combining this estimate with \eqref{eq:sm_em_localization} at \(T=1\)
and then letting \(\rho\uparrow\infty\) gives
\begin{equation}\label{eq:brownian_error_mesh_recovery}
    \|\hat U^n-\bar U^n\|_\infty
    \xrightarrow{\mathbb P}0.
\end{equation}

By \(\|Y-Y^{[n]}\|_\infty\to0\) almost surely and
\eqref{eq:brownian_error_mesh_recovery}, the perturbation
\((Y-Y^{[n]},\hat U^n-\bar U^n)\) tends to zero in probability
uniformly, hence also in the vector-valued \(J_1\)-topology.
By stable Slutsky's theorem, it converges jointly stably with the pair in
\eqref{eq:brownian_grid_time_limit}.
Since the perturbation converges to the continuous zero path,
addition is continuous in the vector-valued \(J_1\)-topology at the
limiting pair. The continuous mapping theorem therefore gives
\[
    (Y,\hat U^n)
    \Rightarrow^{\mathrm{st}}
    (Y,U)
\]
in \(D([0,1],\mathbb R^{q+d})\) with the vector-valued \(J_1\)-topology.
All prelimit and limit paths are continuous, and
the \(J_1\)-topology induced on \(C([0,1],\mathbb R^{q+d})\)
agrees with the uniform topology; hence the same
\(\mathcal F\)-stable convergence holds in the asserted continuous-path space.
\end{proof}

{%
\paragraph{The \(C^2\) case.}
\leavevmode
Assume the hypotheses of \cref{thm:sm_em_joint} and, in addition,
\(f\in C^2(\mathbb R^d;\mathbb R^d)\).
}

{%
For the process \(R^f\) defined in
\eqref{eq:app_augmented_residual_definition},
It\^o's formula identifies
\[
 R_t^{f,i}=\frac12\int_0^t f_{jk}^i(X_{s-})
               \sigma_\alpha^j c_s^{\alpha\beta}\sigma_\beta^k\,ds,
\]
and the residual identity becomes
}
\begin{align}
    f^i(X_t)-f^i(x_0)
    &=
    \int_0^t f_j^i(X_{s-})f^j(X_{s-})\,ds
    +
    \int_0^t f_j^i(X_{s-})\sigma_\alpha^j\,dY_s^\alpha
    \notag\\
    &\quad
    +
    \frac12\int_0^t
    f_{jk}^i(X_{s-})
    \sigma_\alpha^jc_s^{\alpha\beta}\sigma_\beta^k\,ds
    \notag\\
    &\quad
    +
    \int_0^t\int_{\mathbb R^q}
    \Theta_f^i(X_{s-},z)\,\mu^Y(ds,dz).
    \label{eq:detailed_augmented_raw_ito}
\end{align}
{%
Substitution into \eqref{eq:app_augmented_drift_process} gives
\eqref{eq:app_augmented_beta}.
}
\begin{align}
    \beta_s^i
    &=
    f_j^i(X_{s-})f^j(X_{s-})
    +
    f_j^i(X_{s-})\sigma_\alpha^j b_s^\alpha
    \notag\\
    &\quad
    +
    \frac12
    f_{jk}^i(X_{s-})
    \sigma_\alpha^j c_s^{\alpha\beta}\sigma_\beta^k
    \notag\\
    &\quad
    +
    \int_{\mathbb R^q}
    \left\{
        h_f^i\!\left(\Delta_f(X_{s-},z)\right)
        -
        f_j^i(X_{s-})\sigma_\alpha^j h^\alpha(z)
    \right\}F_s(dz).
    \label{eq:app_augmented_beta}
\end{align}

{%
Moreover,
the process \(\mathcal Z\) defined in
\eqref{eq:sm_em_forcing_limit} admits the representation below,
with \(\mathcal I(Y)\) as in \eqref{eq:sm_linear_discretization_companion}:
}
\begin{align}
    \mathcal Z_t^i
    & =
    \frac12
    \int_0^t
    f_j^i(X_{s-})f^j(X_{s-})\,ds
    +
    \int_0^t
    f_j^i(X_{s-})\sigma_\alpha^j\,d\mathcal I_s^\alpha(Y)
    \notag\\
    &\quad
    +
    \frac14
    \int_0^t
    f_{jk}^i(X_{s-})
    \sigma_\alpha^jc_s^{\alpha\beta}\sigma_\beta^k\,ds
    \notag\\
    &\quad
    +
    \int_0^t\int_{\mathbb R^q}\int_0^1
    u\Theta_f^i(X_{s-},z)
    \widetilde\mu^Y(ds,dz,du).
    \label{eq:sm_em_forcing_expanded}
\end{align}
{%
Indeed, insert half of It\^o's formula
\eqref{eq:detailed_augmented_raw_ito} and
\eqref{eq:app_centered_remainder_identification} {%
into \eqref{eq:sm_em_forcing_limit}.
}
The linear stochastic terms combine using
\(\mathcal I(Y)=Y/2+\int_0^\cdot c_s^{1/2}\,dW_s/\sqrt{12}
+{\mathcal K}(Y)\), and the jump remainders combine as
\[
 \tfrac12\Theta_f*\mu^Y
       +\{(u-\tfrac12)\Theta_f\}*\widetilde\mu^Y
       =\{u\Theta_f\}*\widetilde\mu^Y.
\]
The jump integral is absolutely convergent by the local quadratic bound
and the finite large-jump sum.  The coefficient \(1/4\) is the product
of the factor \(1/2\) {%
in \eqref{eq:sm_em_forcing_limit}
} and the factor \(1/2\)
in It\^o's formula.
}

\appendix

\section{Auxiliary Consistency and Linear Stability Results}
\label{app:deterministic_localization_propagation}

\subsection{Consistency and coefficient convergence}

\begin{lemma}[\(L^1\) convergence of grid approximations]
\label{le:cadlag_mesh_L1}
Let \(g:[0,T]\to\mathbb R^m\) be c\`adl\`ag, where \(T\le1\).  Then
\begin{equation}\label{eq:cadlag_mesh_L1_basic}
    \int_0^T
    |g_{s-}-g_{\varphi_n(s)}|\,ds
    \longrightarrow0.
\end{equation}
If \(G:\mathbb R^m\to\mathbb R^\ell\) is continuous, then the same conclusion
holds with \(g\) replaced by \(G\circ g\).
\end{lemma}

\begin{proof}
A c\`adl\`ag path is bounded and continuous outside a countable set.
At every continuity point, \(g_{\varphi_n(s)}\to g_s=g_{s-}\), so
dominated convergence gives \eqref{eq:cadlag_mesh_L1_basic}.
The same argument applies to the c\`adl\`ag path \(G\circ g\).
\end{proof}

\begin{lemma}[Consistency and compact containment of the Euler scheme]
\label{le:sm_em_consistency_localization}
Let \(X\) and \(\hat X^n\) solve \eqref{eq:sm_em_sde} and
\eqref{eq:sm_em_scheme}.  Assume that \(f\) is locally Lipschitz and that the solution of
\eqref{eq:sm_em_sde} is nonexplosive on \([0,1]\).  Then, for every
\(T\le1\),
\begin{equation}\label{eq:sm_em_consistency}
    \sup_{t\le T}|X_t-\hat X_t^n|
    \xrightarrow{\mathbb P}0,
\end{equation}
and
\begin{equation}\label{eq:sm_em_mesh_L1}
    \int_0^T
    |\hat X_{s-}^n-\hat X_{\varphi_n(s)}^n|\,ds
    \xrightarrow{\mathbb P}0.
\end{equation}
Moreover,
\begin{equation}\label{eq:sm_em_localization}
    \lim_{\rho\uparrow\infty}
    \limsup_{n\to\infty}
    \mathbb P
    \left(
        \sup_{t\le T}
        \{|X_t|\vee|\hat X_t^n|\}>\rho
    \right)
    =0.
\end{equation}
\end{lemma}

\begin{proof}
Set \(E^n\coloneqq X-\hat X^n\), which is continuous by
\eqref{eq:sm_em_error_identity}.  Fix \(T\le1\) and \(\rho>0\), and let
\(\tau_\rho\coloneqq\inf\{t:|X_t|>\rho\}\) and
\(\theta_n\coloneqq\inf\{t:|E_t^n|>1\}\), with \(\inf\varnothing=\infty\).
Before these exits, \(X\) and the grid values of \(\hat X^n\) lie in the
ball of radius \(\rho+1\).  If \(L_\rho\) is a Lipschitz constant on
that ball, adding and subtracting \(f(X_{\varphi_n(s)})\) in the error
identity and applying Gronwall gives
\begin{equation}\label{eq:sm_consistency_stopped_error_bound}
    \sup_{t\le T\wedge\tau_\rho\wedge\theta_n}|E_t^n|
    \le
    L_\rho e^{L_\rho T}
    \int_0^T|X_{s-}-X_{\varphi_n(s)}|\,ds.
\end{equation}
The right-hand side tends to zero almost surely by
\cref{le:cadlag_mesh_L1}.  On \(\{\tau_\rho>T\}\), once the right-hand
side is at most \(1/2\), suppose that \(\theta_n\le T\).
By the definition of \(\theta_n\) and continuity of \(E^n\),
\(|E^n_{\theta_n}|=1\), whereas
\eqref{eq:sm_consistency_stopped_error_bound} gives
\(|E^n_{\theta_n}|\le1/2\), a contradiction.  Thus \(\theta_n>T\).
Letting \(n\to\infty\) and then \(\rho\to\infty\) proves
\eqref{eq:sm_em_consistency}, since \(X\) is nonexplosive.
Moreover,
\[
    \int_0^T|\hat X^n_{s-}-\hat X^n_{\varphi_n(s)}|\,ds
    \le \int_0^T|X_{s-}-X_{\varphi_n(s)}|\,ds
       +2T\sup_{t\le T}|E_t^n|,
\]
which proves \eqref{eq:sm_em_mesh_L1}.
Finally, \(\sup_{t\le T}|\hat X_t^n|\le
\sup_{t\le T}|X_t|+\sup_{t\le T}|E_t^n|\) yields compact containment.
\end{proof}

\begin{lemma}[Convergence of the averaged Jacobian]
\label{le:sm_mean_value_coefficient_convergence}
Assume the hypotheses of \cref{le:sm_em_consistency_localization} and
\(f\in C^1(\mathbb R^d;\mathbb R^d)\).
Let \(H^n\) be defined by \eqref{eq:sm_new_mean_value_coefficient}
and set \(H_s=Df(X_{s-})\).  Then
\begin{equation}\label{eq:app_mean_value_coefficient_L1}
    H^n\longrightarrow H
    \quad\text{in probability in }
    L^1([0,T],\mathbb R^{d\times d})
\end{equation}
for every \(T\le1\), and
\begin{equation}\label{eq:app_mean_value_coefficient_tight}
    \|H^n\|_{L^\infty([0,T])}
    =O_{\mathbb P}(1).
\end{equation}

\end{lemma}

\begin{proof}
On \(\Omega_{\rho,n}(T)\coloneqq\{\sup_{t\le T}(|X_t|\vee|\hat X_t^n|)
\le\rho\}\), convexity of the ball gives
\(\|H^n\|_{L^\infty([0,T])}\le\sup_{|x|\le\rho}|Df(x)|\).
Equation \eqref{eq:sm_em_localization} proves
\eqref{eq:app_mean_value_coefficient_tight}.

Define
\[
    \omega_\rho(\eta)\coloneqq
    \sup\{|Df(x)-Df(y)|:|x|,|y|\le\rho,\ |x-y|\le\eta\}.
\]
Uniform continuity on the compact ball gives
\(\omega_\rho(\eta)\to0\) as \(\eta\downarrow0\).
The approximation and mesh errors satisfy, on the same event,
\begin{equation}\label{eq:app_mean_value_split}
\begin{split}
    \|H^n-Df(X_-)\|_{L^1([0,T])}
    &\le T\omega_\rho\!\left(\sup_{t\le T}|X_t-\hat X_t^n|\right)\\
    &\quad+\int_0^T|Df(X_{\varphi_n(s)})-Df(X_{s-})|\,ds.
\end{split}
\end{equation}
By \eqref{eq:sm_em_consistency},
\[
    T\omega_\rho\!\left(\sup_{t\le T}|X_t-\hat X_t^n|\right)
    \xrightarrow{\mathbb P}0.
\]
Apply \cref{le:cadlag_mesh_L1} to \(Df(X)\) for the second term.
Combining these estimates with \eqref{eq:sm_em_localization} and then
letting \(\rho\uparrow\infty\) proves
\eqref{eq:app_mean_value_coefficient_L1}.
\end{proof}

\subsection{Stability of the discretized linear equation}

\begin{lemma}[Stability of the discretized linear equation]
\label{le:grid_linear_propagation}
Let \(\Gamma^n\) be \(\mathbb R^r\)-valued c\`adl\`ag processes, let
\(J^n\) be \(\mathbb R^d\)-valued c\`adl\`ag processes, and let \(H^n\) be
\(\mathbb R^{d\times d}\)-valued measurable processes which belong to
\(L^\infty([0,1])\) almost surely.  Suppose that the
c\`adl\`ag process \(U^n\) satisfies
\begin{equation}\label{eq:app_grid_linear_equation}
    U_t^n
    =
    J_t^n
    +
    \int_0^{\varphi_n(t)}H_s^nU_s^n\,ds,
    \qquad 0\le t\le1.
\end{equation}
Assume that
\begin{equation}\label{eq:app_grid_linear_input_convergence}
    (\Gamma^n,J^n,H^n)
    \Rightarrow
    (\Gamma,J,H)
\end{equation}
in \(D([0,1],\mathbb R^{r+d})
    \times
    L^1([0,1],\mathbb R^{d\times d}),\)
where the first factor carries the usual vector-valued \(J_1\)-topology, and
assume that
\begin{equation}\label{eq:app_grid_linear_coefficient_tightness}
    \|H^n\|_{L^\infty([0,1])}
    =O_{\mathbb P}(1).
\end{equation}
Then
\begin{equation}\label{eq:app_grid_linear_output_convergence}
    (\Gamma^n,U^n)
    \Rightarrow
    (\Gamma,U)
\end{equation}
in \(D([0,1],\mathbb R^{r+d})\), where \(U\) is the unique c\`adl\`ag
solution of
\begin{equation}\label{eq:app_grid_linear_limit_equation}
    U_t
    =
    J_t
    +
    \int_0^tH_sU_s\,ds.
\end{equation}
If \eqref{eq:app_grid_linear_input_convergence} holds stably with respect to
\(\mathcal F\), then
\eqref{eq:app_grid_linear_output_convergence} also holds stably with respect
to \(\mathcal F\).
\end{lemma}

\begin{proof}
For deterministic \(j\in D([0,1],\mathbb R^d)\) and
\(h\in L^1([0,1],\mathbb R^{d\times d})\), consider
\(v_t=j_t+\int_0^t h_sv_s\,ds\).  Picard iteration
\(v^{(0)}\coloneqq j\),
\(v^{(k+1)}_t\coloneqq j_t+\int_0^t h_sv^{(k)}_s\,ds\) satisfies
\[
    \|v^{(k+1)}-v^{(k)}\|_\infty
    \le \|j\|_\infty\frac{\|h\|_{L^1}^{k+1}}{(k+1)!}.
\]
Thus the iterates converge uniformly to a solution.  Its integral part
is absolutely continuous, so the solution is c\`adl\`ag and depends only
on the \(L^1\) class of \(h\).  Gronwall proves uniqueness and gives
\[
    \|v\|_\infty\le\|j\|_\infty e^{\|h\|_{L^1}}.
\]

We show that the solution map \((j,h)\mapsto v\), defined by
\(v_t=j_t+\int_0^t h_sv_s\,ds\), is continuous from
\(J_1\times L^1\) to \(J_1\).
Suppose \(j_n\to j\) in \(J_1\) and \(h_n\to h\) in \(L^1\).
The preceding bound and \(L^1\)-convergence give
\(C\coloneqq\sup_n\|v_n\|_\infty<\infty\).
Set \(i_n\coloneqq v_n-j_n\) and \(i\coloneqq v-j\).  Subtracting the integral equations
and applying Gronwall gives
\begin{equation}\label{eq:app_grid_linear_integral_difference}
    \|i_n-i\|_\infty
    \le\left(C\|h_n-h\|_{L^1}
       +\int_0^1|h_s|\,|j_n(s)-j(s)|\,ds\right)e^{\|h\|_{L^1}}
    \longrightarrow0.
\end{equation}
The integral tends to zero by dominated convergence: \(j_n\) is uniformly
bounded and converges at each continuity point of \(j\), whose
discontinuity set is countable.
Since \(i_n\to i\) uniformly and \(i\) is continuous,
addition gives \(v_n=j_n+i_n\to j+i=v\) in \(J_1\).

Define the joint solution map
\[
    \Psi((\gamma,j),h)\coloneqq(\gamma,v).
\]
The same addition argument applied to \((\gamma_n,j_n)+(0,i_n)\)
proves continuity of
\[
    \Psi:D([0,1],\mathbb R^{r+d})\times L^1
    \longrightarrow D([0,1],\mathbb R^{r+d}),
\]
with both Skorokhod spaces equipped with the vector-valued
\(J_1\)-topology.  In particular, the solution maps are Borel measurable.

Define \(V^n\) as the solution of the continuous-time linear equation
\begin{equation}\label{eq:app_grid_linear_auxiliary_equation}
    V_t^n
    \coloneqq
    J_t^n
    +
    \int_0^tH_s^nV_s^n\,ds.
\end{equation}
The Borel solution map makes \(V^n\) a measurable random element, and
the pathwise bound gives
\begin{equation}\label{eq:app_grid_linear_V_bound}
    \|V^n\|_\infty
    \le \|J^n\|_\infty\exp(\|H^n\|_{L^1}).
\end{equation}
Put \(K_n\coloneqq\|H^n\|_{L^\infty}\).  Subtracting the two equations gives
\[
    U_t^n-V_t^n
    =\int_0^{\varphi_n(t)}H_s^n(U_s^n-V_s^n)\,ds
      -\int_{\varphi_n(t)}^tH_s^nV_s^n\,ds.
\]
A further Gronwall estimate yields
\begin{equation}\label{eq:app_grid_linear_grid_error_bound}
    \|U^n-V^n\|_\infty
    \le\frac{K_n}{n}\|J^n\|_\infty e^{2K_n}
    \xrightarrow{\mathbb P}0.
\end{equation}
Indeed, \(K_n\) is tight by assumption, and \(\|J^n\|_\infty\) is tight
by \eqref{eq:app_grid_linear_input_convergence} and \(J_1\)-continuity of the supremum norm.
On \(\{K_n\le M,\ \|J^n\|_\infty\le M\}\), the bound is at most
\(M^2e^{2M}/n\); first let \(n\to\infty\), then \(M\to\infty\).

By continuity of \(\Psi\) and the continuous mapping theorem,
\eqref{eq:app_grid_linear_input_convergence} gives
\((\Gamma^n,V^n)\Rightarrow(\Gamma,U)\).
Together with \(\|U^n-V^n\|_\infty\to0\) in probability from
\eqref{eq:app_grid_linear_grid_error_bound}, Slutsky's theorem gives
\eqref{eq:app_grid_linear_output_convergence}.

If \eqref{eq:app_grid_linear_input_convergence} holds stably with respect
to \(\mathcal F\), then, for every bounded \(\mathcal F\)-measurable
random variable \(\xi\),
\[
    (\xi,\Gamma^n,J^n,H^n)\Rightarrow(\xi,\Gamma,J,H).
\]
The same continuous mapping and Slutsky argument in the product space gives
\[
    (\xi,\Gamma^n,U^n)\Rightarrow(\xi,\Gamma,U).
\]
Since this holds for every such \(\xi\), the convergence in
\eqref{eq:app_grid_linear_output_convergence} is stable with respect to \(\mathcal F\).
\end{proof}

\section{Proof of the Integrated Discretization Limit}
\label{app:first_integrated_companion}

We prove \cref{thm:integrated_discretization_limit}.  Throughout this
appendix, we use the notation of
Subsection~\ref{subsec:general_integrated_discretization_limit}.

\subsection{Grigelionis representation and regular approximation}

Applying \cite[Theorem~2.1.2]{jacod2012discretization} with Brownian
dimension \(r\) and diffusion coefficient \(a=c^{1/2}\), we obtain a
very good filtered extension
\((\Omega^G,\mathcal G,(\mathcal G_t),\mathbb P^G)\) on which
\begin{align}
    V_t
    &=
    V_0
    +
    \int_0^t b_s\,ds
    +
    \int_0^t a_s\,dW_s^0
    +
    h(\delta)*(p-\mathfrak q)_t
    \notag\\
    &\quad
    +
    \{\delta-h(\delta)\}*p_t.
    \label{eq:app_generic_grigelionis}
\end{align}
Here \(W^0\) is an \(r\)-dimensional Brownian motion,
and \(p\) is a Poisson random measure on \(\mathbb R_+\times E\) with
compensator \(\mathfrak q(ds,de)=ds\,\lambda(de)\), where \(E\) is
Polish and \(\lambda\) is atomless, sigma-finite and infinite.
The predictable map \(\delta\) pushes \(\lambda\), restricted to
\(\{\delta\ne0\}\), onto \(F_s\).
The reference uses the canonical truncation
\(h_0(z)=z\mathbf 1_{\{|z|\le1\}}\); its drift density \(b^0\) and
the drift density \(b\) relative to \(h\) are related by
\begin{equation}\label{eq:detailed_grigelionis_truncation_change}
    b_s
    =
    b_s^0
    +
    \int_E
    \{h(\delta(s,e))-h_0(\delta(s,e))\}\lambda(de).
\end{equation}
The correction is absolutely integrable because \(h-h_0\) is bounded
and vanishes near zero.  Original martingales and characteristics are
preserved on the extension \cite[(2.1.29)]{jacod2012discretization}.

On a further very good extension, independently mark the atoms of
\(p\) by uniform \([0,1]\) variables independent of \(\mathcal G\),
revealing each mark only at the time of its atom.
Let \(\bar p(ds,de,du)\) denote the resulting Poisson random measure on
\(\mathbb R_+\times E\times[0,1]\), with compensator
\(\bar{\mathfrak q}(ds,de,du)=ds\,\lambda(de)\,du\).
Let \(W'\) be an \(r\)-dimensional Brownian motion independent of
\(\mathcal G\) and of the marking.
For notational simplicity, we denote the probability measure on this
further extension by \(\mathbb P\).
The image of \(\bar p\) under \((s,e,u)\mapsto(s,\delta(s,e),u)\), omitting zero
jumps, realizes \(\widetilde\mu^{V}\).  Consequently,
\begin{align}
    \mathcal I_t(V)
    &=
    \frac12\int_0^t b_s\,ds
    +
    \frac12\int_0^t a_s\,dW_s^0
    +
    \frac1{\sqrt{12}}\int_0^t a_s\,dW_s'
    \notag\\
    &\quad
    +
    \{u h(\delta)\}*(\bar p-\bar{\mathfrak q})_t
    +
    \{u(\delta-h(\delta))\}*\bar p_t.
    \label{eq:app_generic_first_companion_grigelionis}
\end{align}
All intermediate stable convergences below are relative to \(\mathcal G\).

Let \(\pi_m(x)=x\) for \(|x|\le m\), and
\(\pi_m(x)=mx/|x|\) otherwise, for vectors or matrices.
Choose measurable \(E_m\uparrow E\) with \(\lambda(E_m)<\infty\),
extend \(a\) by zero to negative times, and set
\begin{equation}\label{eq:app_first_companion_regular_coefficients}
\begin{split}
    b_s^{(m)}&=\pi_m(b_s),\qquad
    a_s^{(m)}=m\int_{(s-1/m)^+}^s\pi_m(a_v)\,dv,\\
    \delta^{(m)}(s,e)&=\delta(s,e)\mathbf 1_{E_m}(e).
\end{split}
\end{equation}
The matrices \(a^{(m)}\) remain symmetric nonnegative, hence are the
symmetric square roots of \(c^{(m)}=a^{(m)}(a^{(m)})^\top\).
Define \(V^{(m)}\) and \(\mathcal I^{(m)}\) on the same extension by
\begin{align}
    V_t^{(m)}
    &\coloneqq
    V_0
    +
    \int_0^t b_s^{(m)}\,ds
    +
    \int_0^t a_s^{(m)}\,dW_s^0
    +
    h(\delta^{(m)})*(p-\mathfrak q)_t
    \notag\\
    &\quad
    +
    \{\delta^{(m)}-h(\delta^{(m)})\}*p_t,
    \label{eq:app_first_companion_regular_process}
\end{align}
\begin{align}
    \mathcal I_t^{(m)}
    &\coloneqq
    \frac12\int_0^t b_s^{(m)}\,ds
    +
    \frac12\int_0^t a_s^{(m)}\,dW_s^0
    +
    \frac1{\sqrt{12}}\int_0^t a_s^{(m)}\,dW_s'
    \notag\\
    &\quad
    +
    \{u h(\delta^{(m)})\}*(\bar p-\bar{\mathfrak q})_t
    +
    \{u(\delta^{(m)}-h(\delta^{(m)}))\}*\bar p_t.
    \label{eq:app_first_companion_regular_limit}
\end{align}

For fixed \(m\), the canonical drift
\[
    b_s^{(m),0}=b_s^{(m)}
       +\int_E\{h_0(\delta^{(m)}(s,e))-h(\delta^{(m)}(s,e))\}\lambda(de)
\]
is bounded.  The backward average \(a^{(m)}\) is adapted, bounded and
continuous, and
\[
    |\delta^{(m)}(s,e)|\wedge1\le\mathbf 1_{E_m}(e),
    \qquad \int_E\mathbf 1_{E_m}(e)^2\lambda(de)=\lambda(E_m)<\infty.
\]
Thus \(V^{(m)}\) satisfies
\cite[Assumption~6.1.1]{jacod2012discretization} with exponent \(2\),
namely (H).  This exponent is unrelated to the dimension \(r\);
no bound on the jump sizes is required.

\begin{lemma}[Regular approximation]
\label{le:first_companion_regular_approximation}
As \(m\to\infty\),
\begin{align}
    &\int_0^1|b_s^{(m)}-b_s|\,ds
    +
    \int_0^1|a_s^{(m)}-a_s|^2\,ds
    \notag\\
    &\qquad
    +
    \int_0^1\int_E
    \bigl(1\wedge|\delta^{(m)}(s,e)-\delta(s,e)|^2\bigr)
    \lambda(de)\,ds
    \longrightarrow0
    \label{eq:app_first_companion_coefficient_approximation}
\end{align}
almost surely.  Moreover, one can write
\(V-V^{(m)}=A^{(m)}+M^{(m)}\),
where \(A^{(m)}\) has finite variation, \(M^{(m)}\) is a locally square
integrable local martingale with continuous predictable bracket, and
\begin{equation}\label{eq:app_first_companion_semimartingale_control}
    \operatorname{Var}(A^{(m)})_1
    +
    \operatorname{tr}\langle M^{(m)}\rangle_1
    \xrightarrow{\mathbb P}0.
\end{equation}
Finally,
\begin{equation}\label{eq:app_first_companion_limit_approximation}
    \sup_{t\le1}
    \left(
        |V_t^{(m)}-V_t|
        +
        |\mathcal I_t^{(m)}-\mathcal I_t(V)|
    \right)
    \xrightarrow{\mathbb P}0.
\end{equation}
\end{lemma}

\begin{proof}
The drift approximation follows by dominated convergence.
For the diffusion coefficient, the contraction property of backward
averaging gives, pathwise,
\[
    \|a^{(m)}-a\|_{L^2}
    \le\|\pi_m(a)-a\|_{L^2}
      +\left\|m\int_{(\,\cdot-1/m\,)^+}^{\,\cdot}a_v\,dv-a\right\|_{L^2}
    \longrightarrow0.
\]
The first term is the truncation error and the second tends to zero by
the approximate-identity property in \(L^2\).
For the jumps,
\(1\wedge|\delta^{(m)}-\delta|^2
=\mathbf 1_{E\setminus E_m}(1\wedge|\delta|^2)\), whose integral tends
to zero by dominated convergence and the image identity
\begin{equation}\label{eq:detailed_delta_integrability_identity}
    \int_0^1\int_E
    (1\wedge|\delta(s,e)|^2)\lambda(de)\,ds
    =
    \int_0^1\int_{\mathbb R^r}
    (1\wedge|z|^2)F_s(dz)\,ds
    <\infty.
\end{equation}

This proves \eqref{eq:app_first_companion_coefficient_approximation}.  Set
\begin{align}
    A_t^{(m)}
    &\coloneqq
    \int_0^t(b_s-b_s^{(m)})\,ds
    \notag\\
    &\quad
    +
    \{\delta-h(\delta)-\delta^{(m)}+h(\delta^{(m)})\}*p_t,
    \label{eq:app_first_companion_Am}\\
    M_t^{(m)}
    &\coloneqq
    \int_0^t(a_s-a_s^{(m)})\,dW_s^0
    +
    \{h(\delta)-h(\delta^{(m)})\}*(p-\mathfrak q)_t.
    \label{eq:app_first_companion_Mm}
\end{align}
Subtracting the Grigelionis decompositions gives
\(V-V^{(m)}=A^{(m)}+M^{(m)}\).
Choose \(\varepsilon_h>0\) such that \(z-h(z)=0\) for
\(|z|\le\varepsilon_h\), and set
\(W\coloneqq\mathbf 1_{\{|\delta|>\varepsilon_h\}}\).
By \eqref{eq:detailed_delta_integrability_identity},
\[
    (W*\mathfrak q)_1
    \le\frac{1}{1\wedge\varepsilon_h^2}
       \bigl((1\wedge|\delta|^2)*\mathfrak q\bigr)_1<\infty.
\]
Hence the counting process \(W*p\), whose compensator is
\(W*\mathfrak q\), has only finitely many jumps on \([0,1]\)
almost surely.
Since the time measure \(ds\) in \(\mathfrak q(ds,de)=ds\,\lambda(de)\)
has no atoms, almost surely no two atoms of \(p\) have the same time.
At each of the finitely many times under consideration, write the
corresponding atom uniquely as \((s,e_s)\); its \(E\)-coordinate is
\(e_s\), and \(\Delta V_s=\delta(s,e_s)\).  Therefore
\begin{align}
    \operatorname{Var}(A^{(m)})_1
    &\le
    \int_0^1|b_s-b_s^{(m)}|\,ds
    \notag\\
    &\quad
    +
    \sum_{s\le1}
    |\Delta V_s-h(\Delta V_s)|
    \mathbf 1_{\{e_s\notin E_m\}}.
    \label{eq:app_first_companion_A_variation}
\end{align}
Here and below, sums involving \(e_s\) run over these finitely many times.
Since \(E_m\uparrow E\), all these coordinates \(e_s\) eventually belong
to \(E_m\).  Thus the variation tends to zero almost surely, without a
moment assumption on the large jumps.  Also,
\begin{align}
    \operatorname{tr}\langle M^{(m)}\rangle_1
    &=
    \int_0^1|a_s-a_s^{(m)}|^2\,ds
    \notag\\
    &\quad
    +
    \int_0^1\int_E
    |h(\delta(s,e))-h(\delta^{(m)}(s,e))|^2
    \lambda(de)\,ds
    \longrightarrow0
    \label{eq:app_first_companion_M_bracket}
\end{align}
since \(|h(\delta)-h(\delta^{(m)})|^2
=\mathbf 1_{E\setminus E_m}|h(\delta)|^2
\le C_h\mathbf 1_{E\setminus E_m}(1\wedge|\delta|^2)\).
The bracket is continuous, which proves
\eqref{eq:app_first_companion_semimartingale_control}.

The local martingales \(M^{(m)}\) start at zero and are locally square
integrable; each coordinate bracket is bounded by the trace in
\eqref{eq:app_first_companion_M_bracket}.  The predictable brackets
converge to zero, so Lenglart's inequality gives \(M^{(m)}\to0\) ucp.
Together with the variation estimate, this yields
\(\|V-V^{(m)}\|_\infty\to0\) in probability.

The processes \(\mathcal I^{(m)}\) use the same \(W^0,W'\) and
\(\bar p\) as \(\mathcal I(V)\).
Their drift and uncompensated jump differences are controlled by the
variation estimate above, since \(0\le u\le1\).
The traces of the predictable brackets at time \(1\) of the two
Brownian-integral differences are respectively \(1/4\) and \(1/12\)
times \(\int_0^1|a_s-a_s^{(m)}|^2\,ds\).
For the compensated marked difference, the corresponding trace is
\(1/3\) times the jump term in \eqref{eq:app_first_companion_M_bracket}.
These martingales start at zero and are locally square integrable;
their coordinate brackets are bounded by these traces, which tend to zero.
Thus Lenglart's inequality proves
\(\|\mathcal I^{(m)}-\mathcal I(V)\|_\infty\to0\) in probability,
and hence \eqref{eq:app_first_companion_limit_approximation}.
\end{proof}

\begin{lemma}[Uniform stability of \(\mathcal I^n\)]
\label{le:first_companion_uniform_stability}
For every \(\eta>0\),
\begin{equation}\label{eq:app_first_companion_uniform_stability}
    \lim_{m\to\infty}
    \sup_{n\ge1}
    \mathbb P\left(
        \sup_{t\le1}
        |\mathcal I_t^n(V)-\mathcal I_t^n(V^{(m)})|>\eta
    \right)
    =0.
\end{equation}
\end{lemma}

\begin{proof}
Put
\[
    \bar\varphi_n(s)\coloneqq \frac{\lceil ns\rceil}{n},
    \qquad
    \vartheta_n(s)\coloneqq n\{\bar\varphi_n(s)-s\},
    \qquad 0\le s\le1.
\]
This deterministic left-continuous weight is predictable, satisfies
\(0\le\vartheta_n\le1\), and equals zero at every grid point \(k/n\).
Cellwise integration by parts against the continuous finite-variation
function \(u\mapsto u\) gives, for a semimartingale \(X\) on a cell of
length \(T\),
\[
    TX_T=\int_0^T X_{u-}\,du+\int_{(0,T]}u\,dX_u.
\]
Consequently, with \(X_u=V_{k/n+u}-V_{k/n}\) and \(T=1/n\),
\[
    n\int_{k/n}^{(k+1)/n}(V_{s-}-V_{k/n})\,ds
    =\int_{(k/n,(k+1)/n]}n\bigl((k+1)/n-s\bigr)\,dV_s.
\]
Summing over the completed cells yields
\begin{equation}\label{eq:app_first_companion_stochastic_integral}
    \mathcal I_t^n(V)
    =
    \int_{(0,\varphi_n(t)]}
    \vartheta_n(s)\,dV_s.
\end{equation}
In particular, a jump exactly at a grid point has weight zero.

Apply this identity to \(V-V^{(m)}=A^{(m)}+M^{(m)}\).
The finite-variation contribution is bounded by
\(\operatorname{Var}(A^{(m)})_1\), uniformly in \(n\), while
\(N^{(m,n)}\coloneqq \int_0^\cdot\vartheta_n(s)\,dM_s^{(m)}\) is a locally square
integrable local martingale starting at zero and satisfies
\(\operatorname{tr}\langle N^{(m,n)}\rangle_1
\le\operatorname{tr}\langle M^{(m)}\rangle_1\).
For an \(\mathbb R^r\)-valued locally square integrable local martingale
\(M\) with \(M_0=0\), applying Lenglart's inequality
to \(|M|^2\), which is
Lenglart-dominated by \(\operatorname{tr}\langle M\rangle\), gives
\begin{equation}\label{eq:app_lenglart_bound}
    \mathbb P(\|M\|_\infty>\eta)
    \le\frac{v}{\eta^2}
       +\mathbb P(\operatorname{tr}\langle M\rangle_1\ge v),
    \qquad \eta,v>0.
\end{equation}
Since
\[
    \sup_{t\le1}
    |N_{\varphi_n(t)}^{(m,n)}|
    \le
    \|N^{(m,n)}\|_\infty,
\]
by \eqref{eq:app_lenglart_bound}, applied to \(N^{(m,n)}\) with threshold
\(\eta/2\),
\begin{align}
    &\sup_{n\ge1}\mathbb P\left(
       \|\mathcal I^n(V)-\mathcal I^n(V^{(m)})\|_\infty>\eta\right)\notag\\
    &\qquad\le
       \mathbb P\left(\operatorname{Var}(A^{(m)})_1>\frac\eta2\right)
       +\frac{4v}{\eta^2}
       +\mathbb P\left(\operatorname{tr}\langle M^{(m)}\rangle_1\ge v\right).
       \notag
\end{align}
Let \(m\to\infty\) using
\eqref{eq:app_first_companion_semimartingale_control}, then \(v\downarrow0\).
\end{proof}

\subsection{The joint limit for the regularized process}

\begin{lemma}[Joint convergence with finitely many jumps]
\label{le:first_companion_grid_synchronization}
Let \(k\ge0\), and suppose that \(c_n\) converges to a continuous path
\(c\) in \(D([0,1],\mathbb R^m)\) with the \(J_1\)-topology.  Let
\[
    0<t_1<\cdots<t_k<1,
    \qquad
    t_{p,n}\longrightarrow t_p,
\]
and let \(v_{p,n},v_p\in\mathbb R^m\) satisfy
\(v_{p,n}\to v_p\).  Define
\[
    x_n(t)
    =
    c_n(t)
    +
    \sum_{p=1}^k
    v_{p,n}\mathbf 1_{[t_{p,n},1]}(t),
\]
and
\[
    x(t)
    =
    c(t)
    +
    \sum_{p=1}^k
    v_p\mathbf 1_{[t_p,1]}(t).
\]
Then \(x_n\to x\) in the vector-valued \(J_1\)-topology.
\end{lemma}

\begin{proof}
Since \(c_n\to c\) in \(J_1\) and \(c\) is continuous, the convergence
is uniform. If \(k=0\), this proves the assertion. For all sufficiently
large \(n\),
\(0<t_{1,n}<\cdots<t_{k,n}<1\).
Let \(\lambda_n\) be the increasing piecewise linear homeomorphism
of \([0,1]\) satisfying
\[
    \lambda_n(0)=0,\qquad \lambda_n(1)=1,\qquad
    \lambda_n(t_p)=t_{p,n},\quad 1\le p\le k.
\]
Then \(\|\lambda_n-\mathrm{id}\|_\infty\to0\), and
\[
\begin{aligned}
    x_n(\lambda_n(t))-x(t)
    &=c_n(\lambda_n(t))-c(t)\\
    &\quad+\sum_{p=1}^k
       (v_{p,n}-v_p)\mathbf1_{[t_p,1]}(t).
\end{aligned}
\]
The right-hand side converges uniformly to zero, since \(c\) is
uniformly continuous, \(c_n\to c\) uniformly, and
\(v_{p,n}\to v_p\) for every \(p\). Thus \(x_n\to x\) in the
vector-valued \(J_1\)-topology.
\end{proof}

\begin{lemma}[Joint limit for the regularized process]
\label{le:first_companion_regular_driver_limit}
For every fixed \(m\),
\begin{equation}\label{eq:app_first_companion_regular_driver_limit}
    \left(
        (V^{(m)})^{[n]},
        \mathcal I^n(V^{(m)})
    \right)
    \Rightarrow^{\mathrm{st}}
    \left(
        V^{(m)},
        \mathcal I^{(m)}
    \right)
\end{equation}
in \(D([0,1],\mathbb R^{2r})\) with the vector-valued \(J_1\)-topology.
The convergence is stable with respect to \(\mathcal G\).
\end{lemma}

\begin{proof}
Fix \(m\) and suppress the superscript \((m)\) on
\(b,a,\delta,V,\mathcal I\) throughout this proof.
The process has finitely many nonzero jumps, since
\(\lambda(E_m)<\infty\).  Absolute continuity in time of the compensator
implies that, almost surely, no jump occurs at any deterministic grid
point, including \(0\) and \(1\).  Write their ordered times and sizes as
\(0<T_1<\cdots<T_N<1\) and \(\zeta_p\coloneqq\Delta V_{T_p}\).
Set \(L_t\coloneqq\sum_{p=1}^N\zeta_p\mathbf 1_{[T_p,1]}(t)\) and
\(C_t\coloneqq V_t-L_t\).  Since
\[
    (|h(\delta)|*\mathfrak q)_t
    \le t\|h\|_\infty\lambda(E_m)<\infty,
\]
we may expand the compensated integral to obtain
\[
    h(\delta)*(p-\mathfrak q)+\{\delta-h(\delta)\}*p
    =\delta*p-h(\delta)*\mathfrak q.
\]
Here \(\delta*p=L\) is a finite pathwise sum, so \(C\) is the continuous
semimartingale
\begin{equation}\label{eq:app_regular_continuous_part}
    C_t
    =
    V_0
    +
    \int_0^t\beta_s^c\,ds
    +
    \int_0^t a_s\,dW_s^0,
    \qquad
    \beta_s^c
    \coloneqq
    b_s-
    \int_E h(\delta(s,e))\,\lambda(de).
\end{equation}
The drift \(\beta^c\) is bounded because \(b\) and \(h\) are bounded and
\(\delta\) is supported on \(E_m\), with \(\lambda(E_m)<\infty\).

For each jump, put \(T_{p,n}\coloneqq\lceil nT_p\rceil/n\) and
\(\kappa_{n,p}\coloneqq n(T_{p,n}-T_p)\).  The exact decompositions are
\begin{align}
    V_t^{[n]}
    &=
    C_t^{[n]}
    +
    \sum_{p=1}^N
    \zeta_p\mathbf 1_{[T_{p,n},1]}(t),
    \label{eq:app_regular_grid_decomposition}\\*
    \mathcal I_t^n(V)
    &=
    \mathcal I_t^n(C)
    +
    \sum_{p=1}^N
    \kappa_{n,p}\zeta_p
    \mathbf 1_{[T_{p,n},1]}(t).
    \label{eq:app_regular_companion_decomposition}
\end{align}
The weight \(\kappa_{n,p}\) is the normalized length of the part of the
cell after \(T_p\), and both contributions appear at \(T_{p,n}\).

On the fixed marked extension, set
\begin{equation}\label{eq:app_regular_continuous_limit_companion}
    \mathcal I_t(C)
    \coloneqq
    \frac12\int_0^t\beta_s^c\,ds
    +
    \frac12\int_0^t a_s\,dW_s^0
    +
    \frac1{\sqrt{12}}\int_0^t a_s\,dW_s'.
\end{equation}
With this ordering, let \(\kappa_p\) denote the already assigned uniform
mark of \(\bar p\) attached to the corresponding atom at \(T_p\).
The bound
\[
    (|u h(\delta)|*\bar{\mathfrak q})_t
    \le\frac t2\|h\|_\infty\lambda(E_m)<\infty
\]
likewise gives
\[
    \{u h(\delta)\}*(\bar p-\bar{\mathfrak q})
    +\{u(\delta-h(\delta))\}*\bar p
    =\{u\delta\}*\bar p-\{u h(\delta)\}*\bar{\mathfrak q}.
\]
Since \(\int_0^1u\,du=1/2\), the continuous drift is
\(\frac12\beta^c\), and hence
\begin{equation}\label{eq:app_regular_limit_decomposition}
    \mathcal I_t
    =
    \mathcal I_t(C)
    +
    \sum_{p=1}^N
    \kappa_p\zeta_p\mathbf 1_{[T_p,1]}(t).
\end{equation}

To apply \cite[Theorem~6.1.2]{jacod2012discretization} jointly, set
\[
    \widehat V_t\coloneqq\begin{pmatrix}V_t\\C_t\end{pmatrix},\qquad
    \widehat a_s\coloneqq\begin{pmatrix}a_s\\a_s\end{pmatrix},\qquad
    \widehat\delta(s,e)\coloneqq
        \begin{pmatrix}\delta(s,e)\\0\end{pmatrix}.
\]
The diffusion coefficient \(\widehat a\) is continuous, and
\(|\widehat\delta|\wedge1\le\mathbf1_{E_m}\).
Recall the canonical truncation \(h_0(z)=z\mathbf1_{\{|z|\le1\}}\).
Using the already defined canonical drift \(b^0\), the augmented drift is
\begin{equation}\label{eq:app_regular_augmented_canonical_drift}
    \widehat b_s^0
    \coloneqq
    \begin{pmatrix}
        b_s^0\\
        \beta_s^c
    \end{pmatrix},
    \qquad
    b_s^0
    =
    b_s
    +
    \int_E
    \{h_0(\delta(s,e))-h(\delta(s,e))\}\,\lambda(de).
\end{equation}
Here Assumption~(H) requires a locally bounded canonical drift,
a c\`adl\`ag diffusion coefficient, and, after localization, a deterministic
upper bound in \(L^2(\lambda)\) for \(|\widehat\delta|\wedge1\).
The canonical drift is bounded, \(\widehat a\) is continuous, and
\[
    |\widehat\delta|\wedge1\le\mathbf1_{E_m},
    \qquad \int_E\mathbf1_{E_m}^2\,d\lambda=\lambda(E_m)<\infty.
\]
Hence \(\widehat V\) satisfies the hypotheses of
Theorem~6.1.2.
Applying that theorem to the identity map with mesh \(1/n\) yields
\begingroup\postdisplaypenalty=10000
\begin{equation}\label{eq:app_regular_augmented_companion_limit}
    \left(
        \mathcal I^n(V),
        \mathcal I^n(C)
    \right)
    \Rightarrow^{\mathrm{st}}
    \left(
        \mathcal I,
        \mathcal I(C)
    \right)
\end{equation}\endgroup
in the vector-valued \(J_1\)-topology, stably with respect to
\(\mathcal G\).  In the limit, the two auxiliary Brownian integrals use the same
\(W'\): before the factor \(1/\sqrt{12}\), their conditional covariance
density is
\[
    \widehat a_s\widehat a_s^\top
    =\begin{pmatrix}a_sa_s^\top&a_sa_s^\top\\
                     a_sa_s^\top&a_sa_s^\top\end{pmatrix}.
\]
The existing \(W'\) is independent of \(\mathcal G\) and realizes this
conditional Gaussian covariance.  At \(T_p\), one uniform mark in the
cited theorem multiplies the whole jump vector \((\zeta_p,0)\).
The already assigned \((\kappa_p)\) have the required conditional uniform
law and independence, both among themselves and from \(W'\).
Thus the limiting process in that theorem
may be realized on the present extension using the existing \(W'\) and
\((\kappa_p)\).  Its two blocks are then precisely \(\mathcal I\) and
\(\mathcal I(C)\), as in \eqref{eq:app_regular_limit_decomposition}.

The fixed linear map \(L(x,y)\coloneqq(y,x-y)\) is continuous on the joint
vector-valued \(J_1\)-space, since
\(\|(Lz_n)\circ\lambda_n-Lz\|_\infty
\le\|L\|_{\mathrm{op}}\|z_n\circ\lambda_n-z\|_\infty\).
By the continuous mapping theorem,
\begin{equation}\label{eq:app_regular_continuous_jump_companion_limit}
    \left(
        \mathcal I^n(C),
        K^n
    \right)
    \Rightarrow^{\mathrm{st}}
    \left(
        \mathcal I(C),
        K
    \right),
    \qquad
    K^n
    \coloneqq
    \mathcal I^n(V)-\mathcal I^n(C),
    \quad
    K
    \coloneqq
    \mathcal I-\mathcal I(C).
\end{equation}
Here \(K_t^n=\sum_{p=1}^N\kappa_{n,p}\zeta_p
\mathbf 1_{[T_{p,n},1]}(t)\) and
\(K_t=\sum_{p=1}^N\kappa_p\zeta_p\mathbf 1_{[T_p,1]}(t)\).

It remains to prove joint convergence of
\((V^{[n]},\mathcal I^n(V))\) to \((V,\mathcal I)\) in the
vector-valued \(J_1\)-topology.
Fix a bounded \(\mathcal G\)-measurable variable \(\xi\) and
\(k\) with \(\mathbb P(N=k)>0\).
The jump data \((T_p,\zeta_p)_{p\le k}\) are \(\mathcal G\)-measurable
on \(\{N=k\}\); give the tuple the fixed value
\(((0,0),\ldots,(0,0))\) on its complement to obtain a globally
defined random element.  This value is irrelevant after multiplication
by \(\mathbf1_{\{N=k\}}\).
Continuity of \(C\) gives \(C^{[n]}\to C\) uniformly almost surely, so the tuple
\[
    (\xi,C^{[n]},(T_p,\zeta_p)_{p\le k},\mathbf 1_{\{N=k\}})
\]
converges in probability. The stable version of Slutsky's theorem
therefore allows us to adjoin this tuple to
\eqref{eq:app_regular_continuous_jump_companion_limit}.
Multiplying bounded continuous test functions by \(\mathbf1_{\{N=k\}}\)
and dividing by \(\mathbb P(N=k)\) gives joint convergence under the
conditional laws given \(N=k\).
Under these conditional laws, the joint convergence of
\(\bigl(\xi,C^{[n]},(\mathcal I^n(C),K^n),(T_p,\zeta_p)_{p\le k}\bigr)\)
is in the Polish space
\[
    \mathbb R\times D([0,1],\mathbb R^r)
    \times D([0,1],\mathbb R^{2r})
    \times(\mathbb R\times\mathbb R^r)^k,
\]
equipped with the product topology, where each Skorokhod space carries
the \(J_1\)-topology.
By the Skorokhod representation theorem, applied to this finite product
of Polish spaces, we may represent
the joint random elements so that they converge almost surely.
Denote the jump data in the \(n\)th represented copy by
\((T_p^n,\zeta_p^n)\), with limit \((T_p,\zeta_p)\).
On the Skorokhod representation space, these jump times and sizes
converge in Euclidean space, \((\mathcal I^n(C),K^n)\) converges in the
vector-valued \(J_1\)-topology, and \(C^{[n]}\) converges uniformly.  For simplicity, we use the same
notation for the represented variables until we return to the original
conditional laws.

The maps
\[
    (x,t)\mapsto x(t),\qquad (x,t)\mapsto x(t-),\qquad
    (x,t)\mapsto\Delta x_t
\]
are Borel on \(D([0,1],\mathbb R^r)\times[0,1]\), with \(x(0-)=x(0)\).
The map
\[
    ((t_p,z_p)_{p\le k})\longmapsto
    \sum_{p=1}^k z_p\mathbf1_{[t_p,1]}
\]
and its mesh-shifted version, with \(t_p\) replaced by
\(\lceil nt_p\rceil/n\), are also Borel.  The identities above are equalities between Borel measurable functions
of the joint random element. Since they hold almost surely for the original
variables, equality of the joint laws implies that they also hold almost
surely on the Skorokhod representation space.
Hence, putting
\[
    S_{p,n}\coloneqq\frac{\lceil nT_p^n\rceil}{n},\qquad
    q_{p,n}\coloneqq n(S_{p,n}-T_p^n)\zeta_p^n,\qquad
    q_p\coloneqq\Delta K_{T_p},
\]
we have \(K_t^n=\sum_{p=1}^kq_{p,n}\mathbf 1_{[S_{p,n},1]}(t)\)
and \(K_t=\sum_{p=1}^kq_p\mathbf 1_{[T_p,1]}(t)\).
We have \(|S_{p,n}-T_p|\le n^{-1}+|T_p^n-T_p|\to0\).
For each \(p\), choose continuity points
\(a_p<T_p<b_p\) enclosing no other limiting jump time.  Eventually the
interval contains only the possible jump \(S_{p,n}\) of \(K^n\), so
\begin{equation}\label{eq:app_regular_jump_size_recovery}
    q_{p,n}=K^n_{b_p}-K^n_{a_p}
    \longrightarrow K_{b_p}-K_{a_p}=q_p.
\end{equation}
Hence \(q_{p,n}\to q_p\) for every \(p\le k\).

Since \(C\) and \(\mathcal I(C)\) are continuous,
\((C^{[n]},\mathcal I^n(C))\) converges uniformly to
\((C,\mathcal I(C))\).
Apply \cref{le:first_companion_grid_synchronization} with times
\(S_{p,n}\to T_p\) and vector jump sizes
\((\zeta_p^n,q_{p,n})\to(\zeta_p,q_p)\).
The exact decompositions reconstruct the copies of
\((V^{[n]},\mathcal I^n(V))\) and \((V,\mathcal I)\), yielding joint
convergence in \(D([0,1],\mathbb R^{2r})\) with the vector-valued
\(J_1\)-topology, jointly with the corresponding copies of \(\xi\).

These copies of \(\xi\) are bounded by \(\|\xi\|_\infty\), so dominated
convergence with a bounded continuous test functional proves the stable
test-function limit under the conditional laws on \(\{N=k\}\).
Summing over \(k\le K\) and bounding the tail by
\(2\|\xi\|_\infty\|\Phi\|_\infty\mathbb P(N>K)\), for a bounded
continuous test functional \(\Phi\), proves the \(\mathcal G\)-stable
limit.  Restoring \((m)\) yields
\eqref{eq:app_first_companion_regular_driver_limit}.
\end{proof}

\subsection{Removal of the regularization}

\begin{proof}[Proof of \cref{thm:integrated_discretization_limit}]
Apply \cite[Proposition~2.2.4]{jacod2012discretization} in
\(D([0,1],\mathbb R^{2r})\) with its vector-valued \(J_1\)-topology.
For fixed \(m\), the joint \(\mathcal G\)-stable limit is
\cref{le:first_companion_regular_driver_limit}.
The approximation is uniform in \(n\) by
\cref{le:first_companion_regular_approximation,le:first_companion_uniform_stability},
using
\[
    \|V^{[n]}-(V^{(m)})^{[n]}\|_\infty
    \le\|V-V^{(m)}\|_\infty.
\]
The processes \((V^{(m)},\mathcal I^{(m)})\) converge uniformly
in probability to \((V,\mathcal I(V))\) on the same very good filtered
extension by \eqref{eq:app_first_companion_limit_approximation}.
Since the convergence holds in probability on the common extension,
it is stable with respect to the full sigma-field of that extension,
and hence also with respect to \(\mathcal G\).
The cited proposition proves the desired joint \(\mathcal G\)-stable
limit, which is \(\mathcal F\)-stable because
\(\mathcal F\subseteq\mathcal G\).
\end{proof}

\section[Occupation Estimates and a Sobolev It\^o Formula for L\'evy-Driven SDEs]{Occupation Estimates and a Sobolev It\^o Formula for L\'evy-Driven SDEs}
\label{app:levy_smoothing}

This appendix verifies the hypotheses of \cref{prop:sm_euler_transfer}
under condition (B) of \cref{thm:sm_em_relaxed_regularity}.  Throughout, \(Y\) is an
\(\mathbb F\)-L\'evy process with characteristics \((b,c,F)\) relative
to the fixed truncation \(h\), and \(L=\sigma Y\) satisfies
\eqref{eq:levy_smoothing_K} for \(1<p<\infty\).
We use \(p'=p/(p-1)\) and set
\(Q=\sigma c\sigma^\top\), with entries
\(Q^{jk}=\sigma_\alpha^j c^{\alpha\beta}\sigma_\beta^k\).
For a vector-valued Sobolev function \(v\),
\(\partial_{jk}v\) denotes its componentwise weak second derivative,
using a Borel representative.

For background on transition densities of L\'evy processes, Knopova and Schilling
\cite{knopova2013note} study criteria for the existence and smoothness
of L\'evy transition densities, while Schilling, Sztonyk and Wang
\cite{schilling2012coupling} obtain symbol-based gradient estimates
for L\'evy semigroups.  For stable noise, Bogdan and Jakubowski
\cite{bogdan2007estimates} construct and estimate the heat kernel
under gradient perturbations with a Kato-class drift.
The occupation estimate needed here is proved directly from
\eqref{eq:levy_smoothing_K} below.

\subsection[Occupation estimates]{Occupation estimates}
\label{subsec:levy_occupation_estimates}

\begin{lemma}[Occupation estimate for a bounded drift]
\label{lem:levy_smoothing_occupation}
Let \(g:\mathbb R^d\to\mathbb R^d\) be bounded and globally Lipschitz,
and let
\[
 X_t=x_0+\int_0^t g(X_{s-})ds+L_t.
\]
For every \(0<t\le1\), the law of \(X_t\) has a density.  A jointly
measurable choice \(\rho\) satisfies
\begin{equation}\label{eq:levy_density_space}
 \rho\in L^\infty((0,1);L^1(\mathbb R^d))
          \cap L^1((0,1);L^{p'}(\mathbb R^d)).
\end{equation}
In particular, for some finite constant \(C_g\) and all
\(v\in L^p(\mathbb R^d)\),
\begin{equation}\label{eq:levy_smoothing_occupation}
 \sup_{a\in\mathbb R^d}\mathbb E\int_0^1|v(X_s+a)|ds
          \le C_g\|v\|_{L^p}.
\end{equation}
The same estimate holds with \(X_{s-}\) in place of \(X_s\).
\end{lemma}

\begin{proof}
Let \(P_t\varphi(x)=\mathbb E[\varphi(x+L_t)]\), and let
\(\mathcal A\) be its generator.
We use the standard properties of L\'evy semigroups and their
generators (see \cite[Chapter~3]{applebaum2009levy}).
For the semigroup of \(-L\),
\[
 \check P_r\psi(x)=\mathbb E[\psi(x-L_r)],\qquad
 k_t=\check P_{t-a}k_a,\qquad
 \partial_jk_t=\check P_{t-a}(\partial_jk_a),\quad t>a>0.
\]
The convolution semigroup of \(-L\) is strongly continuous on
\(L^1\) and \(L^{p'}\). Thus the identities above give strong
continuity of \(t\mapsto k_t\) in both spaces and of
\(t\mapsto\partial_jk_t\) in \(L^1\), for \(t>0\).
Choose jointly measurable versions.
The convolution integrals below take values in the separable space
\(L^1\), and their measurability and integrable norm bounds ensure
that the Bochner integrals are well defined.

The free semigroup \(P\) is a strongly continuous Feller contraction
semigroup.
For \(\varphi\in C_c^\infty\subset D(\mathcal A)\), its generator
identities give
\[
 \frac{d}{dr}P_r\varphi
   =\mathcal A P_r\varphi=P_r\mathcal A\varphi
\]
in \(C_0\).
For \(|\gamma|\le2\),
\(D^\gamma P_r\varphi=P_rD^\gamma\varphi\), with bounded derivatives.
The second-order Taylor estimate for small jumps and the finiteness of
the L\'evy measure outside the unit ball justify Fubini's theorem for
the L\'evy integral.  Consequently, the L\'evy--Khintchine expression
at \(P_r\varphi\) equals \(P_r\mathcal A\varphi=\mathcal AP_r\varphi\).
Fix \(0<t\le1\), a real-valued \(\varphi\in C_c^\infty\), and
\(\varepsilon>0\), and put
\[
 u_\varepsilon(s,x)=P_{t+\varepsilon-s}\varphi(x),\qquad 0\le s\le t.
\]
Since \(\varphi\in C_c^\infty\), its Fourier transform decays
faster than any power. Together with the at most quadratic growth
of the L\'evy symbol, the Fourier representation of the semigroup
shows that \(u_\varepsilon\) is jointly \(C^2\) in time and space
on a neighborhood of \([0,t]\times\mathbb R^d\).  Differentiation under the semigroup
and its contraction property give
\[
 \partial_su_\varepsilon+\mathcal Au_\varepsilon=0,\qquad
 \|u_\varepsilon\|_\infty\le\|\varphi\|_\infty,\qquad
 \|\nabla u_\varepsilon\|_\infty\le\|\nabla\varphi\|_\infty.
\]
This identifies \(\mathcal Au_\varepsilon\) with the
finite-variation term arising from the L\'evy part in It\^o's formula.
Applying It\^o's formula to \((s,X_s)\) and using the generator
of \(L\), we obtain the local martingale
\[
 N_s^\varepsilon=u_\varepsilon(s,X_s)-u_\varepsilon(0,x_0)
   -\int_0^s g(X_{r-})\cdot\nabla u_\varepsilon(r,X_{r-})\,dr.
\]
Let \(\nu_L\) be the L\'evy measure of \(L\) and set
\(H_\varepsilon(s,z)=u_\varepsilon(s,X_{s-}+z)-u_\varepsilon(s,X_{s-})\).
The preceding bounds give
\[
 |H_\varepsilon(s,z)|
 \le\|\nabla\varphi\|_\infty|z|\mathbf1_{\{|z|\le1\}}
       +2\|\varphi\|_\infty\mathbf1_{\{|z|>1\}},
\]
and hence
\[
 \mathbb E\int_0^t\int_{\mathbb R^d}
       |H_\varepsilon(s,z)|^2\nu_L(dz)\,ds
 \le t\left(\|\nabla\varphi\|_\infty^2
                  \int_{|z|\le1}|z|^2\nu_L(dz)
       +4\|\varphi\|_\infty^2\nu_L(\{|z|>1\})\right)<\infty.
\]
The continuous martingale integral is square-integrable as well,
since its expected quadratic variation is at most
\(t\|Q\|_{\mathrm{op}}\|\nabla\varphi\|_\infty^2\).
Thus both martingale integrals are well defined, and \(N^\varepsilon\)
is a local martingale.
Since
\[
 \sup_{s\le t}|N_s^\varepsilon|
 \le2\|\varphi\|_\infty
      +t\|g\|_\infty\|\nabla\varphi\|_\infty,
\]
\(N^\varepsilon\) is a true martingale.
Take expectations and let \(\varepsilon\downarrow0\).
The uniform convergence of \(P_\varepsilon\varphi\) to \(\varphi\) and
of \(P_\varepsilon\nabla\varphi\) to \(\nabla\varphi\), together with the
contraction property and \(X_s=X_{s-}\) for almost every \(s\), gives
\begin{equation}\label{eq:levy_smoothing_weak_duhamel}
 \mathbb E\varphi(X_t)
 =P_t\varphi(x_0)
 +\int_0^t
   \mathbb E\bigl[
      g(X_s)\cdot\nabla P_{t-s}\varphi(X_s)
   \bigr]\,ds.
\end{equation}

Write \(\mu_s=\mathcal L(X_s)\), \(M=\max_j\|g^j\|_\infty\), and
\(K(r)=\sum_{j=1}^d\|\partial_jk_r\|_{L^1}\).
Integration by parts in the convolution kernel gives the minus sign
in the following \(L^1\) function:
\begin{equation}\label{eq:levy_smoothing_actual_density}
 q_t=k_t(\cdot-x_0)
       -\sum_{j=1}^d\int_0^t
                   (\partial_jk_{t-s})*(g^j\mu_s)\,ds.
\end{equation}
For clarity, convolution with the signed measure \(g^j\mu_s\) means
\[
 [\partial_jk_{t-s}*(g^j\mu_s)](x)
       =\mathbb E[g^j(X_s)\partial_jk_{t-s}(x-X_s)].
\]
It is jointly measurable, and its \(L^1\) norm is bounded by
\(M\|\partial_jk_{t-s}\|_1\).  Since \(K\in L^1(0,1)\),
\eqref{eq:levy_smoothing_actual_density} is a well-defined Bochner
integral in \(L^1\) for each \(0<t\le1\).
Equation \eqref{eq:levy_smoothing_weak_duhamel} implies
\(\mu_t=q_t(x)dx\), first on \(C_c^\infty\) and hence as finite
measures.  Thus \(q_t\ge0\) almost everywhere and \(\|q_t\|_1=1\).
This proves the existence of a density before any \(L^{p'}\) estimate.

We next prove that its \(L^{p'}\)-norm is
integrable in time.
On the Banach space
\[
 \mathcal X=L^\infty((0,1);L^1)
                   \cap L^1((0,1);L^{p'}),
\]
use the equivalent norm
\[
 \|v\|_{\mathcal X,\lambda}
 =\operatorname*{ess\,sup}_{0<t<1}e^{-\lambda t}\|v_t\|_1
        +\int_0^1e^{-\lambda t}\|v_t\|_{p'}dt.
\]
Define
\[
 (\mathcal Tv)_t=k_t(\cdot-x_0)
        -\sum_j\int_0^t(\partial_jk_{t-s})*(g^jv_s)ds,
 \qquad K_\lambda=\int_0^1e^{-\lambda r}K(r)dr.
\]
For \(v\in\mathcal X\) and \(r=1,p'\), Young's inequality gives
\[
 \sum_j\|(\partial_jk_{t-s})*(g^jv_s)\|_{L^r}
        \le MK(t-s)\|v_s\|_{L^r}.
\]
These integrands are strongly measurable in both spaces, by
the strong continuity of \(\partial_jk\) in \(L^1\) and continuity
of convolution.  Since \(K\in L^1(0,1)\), the \(L^1\)-valued
Bochner integral is defined for every \(0<t\le1\), and the
\(L^{p'}\)-valued integral for almost every \(t\).  In fact, Tonelli's
theorem gives
\[
 \begin{aligned}
 \operatorname*{ess\,sup}_{0<t<1}
   \sum_j\int_0^t\|(\partial_jk_{t-s})*(g^jv_s)\|_1\,ds
 &\le M\|K\|_{L^1(0,1)}\|v\|_{L^\infty((0,1);L^1)},\\
 \int_0^1\sum_j\int_0^t
   \|(\partial_jk_{t-s})*(g^jv_s)\|_{p'}\,ds\,dt
 &\le M\|K\|_{L^1(0,1)}\|v\|_{L^1((0,1);L^{p'})}.
 \end{aligned}
\]
Together with \(\|k_t\|_1=1\) and the assumed integrability of
\(\|k_t\|_{p'}\), these bounds show that \(\mathcal T\) maps
\(\mathcal X\) into itself.
Young's inequality and Tonelli's theorem, applied separately to the
two terms in this norm, give
\begin{equation}\label{eq:levy_smoothing_contraction}
 \|\mathcal Tv-\mathcal Tw\|_{\mathcal X,\lambda}
         \le MK_\lambda\|v-w\|_{\mathcal X,\lambda}.
\end{equation}
The free term belongs to \(\mathcal X\) by
\eqref{eq:levy_smoothing_K}.  Moreover \(K_\lambda\to0\) as
\(\lambda\to\infty\).  Choose \(\lambda\) with \(MK_\lambda<1\).
The contraction principle provides a unique fixed point
\(\rho\in\mathcal X\).

The same contraction estimate in
\(L^\infty((0,1);L^1)\), without its \(L^{p'}\) term, gives
uniqueness in that larger space.  By \eqref{eq:levy_smoothing_actual_density}, the family
\(q=(q_t)_{0<t\le1}\) of densities of the laws of \(X_t\) solves this
equation and satisfies \(\|q_t\|_1=1\).
Hence \(q=\rho\) for almost every time.  This identifies the fixed
point with the probability law and proves
\eqref{eq:levy_density_space}, without assuming the existence of a density
or positivity in the fixed-point argument.
We take the jointly measurable family \((q_t)_{0<t\le1}\) as a version
of \(\rho\). Each \(q_t\) is a density of the law of \(X_t\), while
the \(L^{p'}\) assertion holds only for almost every time.
Finally, H\"older's inequality and translation invariance of the \(L^p\) norm yield
\[
 \mathbb E\int_0^1|v(X_s+a)|ds
 \le \|v\|_p\int_0^1\|\rho_s\|_{p'}ds
\]
uniformly in \(a\).  The left-limit version follows because
\(X_s=X_{s-}\) for Lebesgue-almost every \(s\), pathwise.
\end{proof}

\paragraph[Gaussian and \texorpdfstring{\(\alpha\)}{alpha}-stable transition-density bounds.]{Gaussian and \(\alpha\)-stable transition-density bounds.}
The bounds \eqref{eq:sm_stable_smoothing_bounds} can be checked without
a density theorem for the drifted equation.  For Brownian motion,
\(\|k_t\|_{p'}=Ct^{-d/(2p)}\) and
\(\|\nabla k_t\|_1=Ct^{-1/2}\) follow from the Gaussian density.
For \(1<\alpha<2\), represent a rotationally symmetric stable process
as Brownian motion subordinated by an independent
\(\alpha/2\)-stable subordinator \(T\); see
\cite[Example~1.3.28 and Theorem~3.3.15(1)]{applebaum2009levy},
where the subordinator index is our \(\alpha/2\).
Normalize \(\mathbb E[e^{-uT_t}]=e^{-tu^{\alpha/2}}\) and take
Brownian covariance \(2I\); other scales change only the constants.
All negative moments of \(T_1\) are finite, since for \(r>0\),
\[
 \mathbb E[T_1^{-r}]
   =\frac1{\Gamma(r)}\int_0^\infty
                 u^{r-1}e^{-u^{\alpha/2}}du<\infty.
\]
The Gaussian bounds, Minkowski's inequality, and
\(T_t\stackrel{d}=t^{2/\alpha}T_1\) give
\(\|k_t^{(\alpha)}\|_{p'}\le Ct^{-d/(\alpha p)}\) and
\(\|\nabla k_t^{(\alpha)}\|_1\le Ct^{-1/\alpha}\).
An invertible linear change of coordinates changes only the constants.
Convolution with the probability law of an independent \(J_t\) does
not increase either norm, by Young's inequality.  The bounds are
integrable in time when \(\alpha>1\) and
\(p>\max\{1,d/\alpha\}\), as asserted in
Subsection~\ref{sec:regularity_refinements_examples}.

\subsection{A Sobolev It\^o formula}
\label{subsec:levy_sobolev_ito}

\begin{lemma}[Sobolev It\^o formula]
\label{lem:levy_smoothing_ito}
Under the hypotheses of \cref{lem:levy_smoothing_occupation}, let
\(v\in C_c^1(\mathbb R^d;\mathbb R^m)
                  \cap W^{2,p}(\mathbb R^d;\mathbb R^m)\) and put
\(\Theta_v(x,z)=v(x+\sigma z)-v(x)-Dv(x)\sigma z\).
Then, indistinguishably on \([0,1]\),
\begin{equation}\label{eq:levy_smoothing_sobolev_ito}
 v(X_t)-v(x_0)
   =\int_0^t Dv(X_{s-})dX_s
      +\frac12\int_0^t Q^{jk}\partial_{jk}v(X_{s-})\,ds
      +\sum_{0<s\le t}\Theta_v(X_{s-},\Delta Y_s).
\end{equation}
The Hessian term is absolutely integrable in time, and the series
\(\sum_{0<s\le1}\Theta_v(X_{s-},\Delta Y_s)\) converges absolutely.  Moreover,
\begin{equation}\label{eq:levy_smoothing_remainder_integrability}
 \mathbb E\int_0^1\int_{|z|\le1}
                 |\Theta_v(X_{s-},z)|F(dz)ds<\infty.
\end{equation}
Any Borel representative of the weak Hessian gives the same time integral.
\end{lemma}

\begin{proof}
Put \(M_{2,\sigma}=\int_{|z|\le1}|\sigma z|^2F(dz)<\infty\).
We first establish a remainder estimate for every
\(w\in C_c^1\cap W^{2,p}\).
For smooth \(w=(w^1,\ldots,w^m)\), Taylor's formula gives, for \(1\le a\le m\),
\[
 \Theta_w^a(x,z)
 =
 \int_0^1(1-u)
 \sum_{j,k=1}^d
 \partial_{jk}w^a(x+u\sigma z)
 (\sigma z)^j(\sigma z)^k\,du.
\]
Set
\[
 H_w(x)
 \coloneqq
 \left(
 \sum_{a=1}^m\sum_{j,k=1}^d
 |\partial_{jk}w^a(x)|^2
 \right)^{1/2}.
\]
By the Cauchy--Schwarz inequality,
\[
 |\Theta_w(X_{s-},z)|
 \le |\sigma z|^2\int_0^1(1-u)H_w(X_{s-}+u\sigma z)\,du.
\]
For fixed \(u\) and \(z\), the shift \(u\sigma z\) is
deterministic, so the occupation estimate gives
\[
 \mathbb E\int_0^1 H_w(X_{s-}+u\sigma z)\,ds
              \le C_g\|D^2w\|_{L^p}.
\]
The factor \(\int_0^1(1-u)\,du=1/2\) is absorbed in the
constant below.
\begingroup\interdisplaylinepenalty=10000 % Keep (C.7) on one page.
Tonelli's theorem and the compensation formula now yield
\begin{align}
 &\mathbb E\int_0^1\int_{|z|\le1}
               |\Theta_w(X_{s-},z)|F(dz)\,ds
   =\mathbb E\sum_{\substack{0<s\le1\\0<|\Delta Y_s|\le1}}
               |\Theta_w(X_{s-},\Delta Y_s)|\notag\\
 &\hspace{20mm}\le C M_{2,\sigma}\|D^2w\|_{L^p}.
 \label{eq:levy_smoothing_remainder_bound}
\end{align}
\endgroup
Here compensation is valid for nonnegative predictable integrands,
initially with values in \([0,\infty]\).
For general \(w\), take \(w_n=w*\varrho_n\).
By the standard properties of mollification,
\(w_n,Dw_n\) converge uniformly to \(w,Dw\), and
\(w_n\to w\) in \(W^{2,p}\): the local convergence statements
become global because the functions and their mollifications are
all supported in a common compact set.
Fatou's lemma extends \eqref{eq:levy_smoothing_remainder_bound}
to \(w\), with equality again supplied by compensation.
In particular, it proves \eqref{eq:levy_smoothing_remainder_integrability}.

Now take \(v_n=v*\varrho_n\) and apply It\^o's formula
componentwise to \(v_n(X)\).
As above, \(v_n,Dv_n\) converge uniformly to \(v,Dv\), and
\(D^2v_n\to D^2v\) in \(L^p\).
The left-hand sides converge uniformly.  The integrands
\(Dv_n(X_-)\) are predictable and bounded by a common deterministic
constant, so their stochastic integrals converge componentwise in
ucp by the dominated convergence theorem for stochastic integrals.
For the Hessian term, \eqref{eq:levy_smoothing_occupation} gives
\[
 \mathbb E\int_0^1
       |Q^{jk}(\partial_{jk}v_n-\partial_{jk}v)(X_{s-})|ds
       \le C\|Q\|\|D^2v_n-D^2v\|_p\longrightarrow0.
\]
Applying the general estimate \eqref{eq:levy_smoothing_remainder_bound}
directly to \(w=v_n-v\) gives convergence of the small-jump
remainders in expected total variation:
\[
 \mathbb E\sum_{\substack{0<s\le1\\0<|\Delta Y_s|\le1}}
   |\Theta_{v_n}(X_{s-},\Delta Y_s)-\Theta_v(X_{s-},\Delta Y_s)|
 \le C M_{2,\sigma}\|D^2(v_n-v)\|_p\longrightarrow0.
\]
There are almost surely finitely many jumps with
\(|\Delta Y_s|>1\); their remainders converge pathwise by the
uniform convergence of \(v_n,Dv_n\).
The Hessian integral is absolutely integrable in time by \eqref{eq:levy_smoothing_occupation},
and the jump sum is absolutely convergent by
\eqref{eq:levy_smoothing_remainder_bound} and finiteness of the
large-jump sum.  Passing to the limit in ucp proves
\eqref{eq:levy_smoothing_sobolev_ito} indistinguishably on \([0,1]\).
Finally, applying the occupation estimate to the indicator of a
Lebesgue-null set proves independence of the Borel representative
of the weak Hessian.
\end{proof}

\subsection{Characteristics and localization}
\label{subsec:levy_characteristics_localization}

\begin{lemma}[Characteristics of \(f(X)\)]
\label{le:levy_smoothing_local_characteristics}
Assume the hypotheses of \cref{thm:sm_em_relaxed_regularity}, with condition (B). Then
\eqref{eq:app_augmented_fX_decomposition} and
\eqref{eq:sm_common_calculus_integrability} hold, with
\begin{equation}\label{eq:levy_smoothing_beta}
 \beta_s=Df(X_{s-})\{f(X_{s-})+\sigma b\}
          +\frac12Q^{jk}\partial_{jk}f(X_{s-})
          +\int_{\mathbb R^q}r_s(z)F(dz),
\end{equation}
where \(r\) is defined by \eqref{eq:sm_common_truncated_correction}.
The right-hand side defines \(\beta\) outside a predictable
\(ds\otimes d\mathbb P\)-null set; we set \(\beta=0\) on that set.
\end{lemma}

\begin{proof}
For an integer \(R>|x_0|\), choose \(\chi_R\in C_c^\infty\) equal
to one on a neighborhood of \(\overline B(0,R+2)\), and put
\(f_R=\chi_Rf\).  The Leibniz rule gives
\(f_R\in C_c^1\cap W^{2,p}\); in particular, \(f_R\) is bounded
and globally Lipschitz.
Since \(f_R\) is bounded and globally Lipschitz, the cutoff
equation driven by the same \(Y\) has a unique global solution
\(X^R\). Indeed, subtracting \(\sigma Y\) reduces it to an ordinary
integral equation, to which Picard iteration applies pathwise.
Set
\[
 \tau_R=\inf\{t\in[0,1]:|X_t|>R\}\wedge1.
\]
Since the two equations are driven by the same L\'evy process
and \(f=f_R\) on \(B(0,R+2)\), their stopped difference satisfies
\[
 X_{t\wedge\tau_R}-X^R_{t\wedge\tau_R}
 =\int_0^{t\wedge\tau_R}
       \bigl[f_R(X_{s-})-f_R(X^R_{s-})\bigr]\,ds.
\]
The global Lipschitz property of \(f_R\) and Gronwall's lemma
imply \(X_{t\wedge\tau_R}=X^R_{t\wedge\tau_R}\) for all
\(t\in[0,1]\).  This includes the exit jump, since the additive
noise cancels in the difference.
Apply \cref{lem:levy_smoothing_occupation,lem:levy_smoothing_ito}
to the global solution \(X^R\), and stop the identity
\eqref{eq:levy_smoothing_sobolev_ito} for \(f_R(X^R)\) at \(\tau_R\).

Write \(D_R=f(X_{\tau_R})-f_R(X_{\tau_R})\).
The functions \(f_R\) and \(f\) agree before exit, so, for every \(t\in[0,1]\),
\begin{align}
 f(X_{t\wedge\tau_R})-f_R(X^R_{t\wedge\tau_R})
 &=\mathbf1_{\{\tau_R\le t\}}D_R,
 \label{eq:levy_localization_exit_left}\\*
 \sum_{0<s\le t\wedge\tau_R}
 \bigl[\Theta_f(X_{s-},\Delta Y_s)
       -\Theta_{f_R}(X^R_{s-},\Delta Y_s)\bigr]
 &=\mathbf1_{\{\tau_R\le t\}}D_R.
 \label{eq:levy_localization_exit_jumps}
\end{align}
Indeed, the remainders agree at every jump before \(\tau_R\), and
at an exit jump their difference is \(D_R\).
If \(\Delta X_{\tau_R}=0\), then
\(X_{\tau_R}=X_{\tau_R-}\in\overline B(0,R)\), so \(D_R=0\).
Thus a nonzero correction occurs only at an actual jump included in
the second identity; this also covers \(\tau_R=1\).
The first-order integrals agree because their stopped integrands
and integrators agree.  The weak Hessians agree almost everywhere
on \(B(0,R+2)\), and the occupation estimate for \(X^R\), restricted
to \([0,\tau_R]\), makes their Hessian integrals agree, independently
of the representatives.
The two corrections in
\eqref{eq:levy_localization_exit_left} and \eqref{eq:levy_localization_exit_jumps}
therefore cancel, giving the corresponding stopped It\^o formula for \(f(X)\).
Here \(D_R\) is only a finite pathwise quantity.

For \(s\le\tau_R\) and \(|z|\le\varepsilon_\sigma\), both
\(X_{s-}\) and \(X_{s-}+\sigma z\) lie in \(B(0,R+2)\), where \(f_R=f\).  Since both truncations equal the identity near zero and
\(Df\) is bounded on \(\overline B(0,R+2)\), choose a deterministic
\(0<\delta_R\le\varepsilon_\sigma\) such that
\[
 r_s(z)=\Theta_f(X_{s-},z)=\Theta_{f_R}(X^R_{s-},z),
 \qquad s\le\tau_R,\quad |z|\le\delta_R.
\]
For \(|z|\le\delta_R\), use \eqref{eq:levy_smoothing_remainder_bound}
for \(f_R\) and \(X^R\).  For \(|z|>\delta_R\),
\[
 |r_s(z)|\le\|h_f\|_\infty+
     \sup_{|x|\le R}\|Df(x)\|_{\mathrm{op}}
               \|\sigma\|_{\mathrm{op}}\|h\|_\infty,
 \qquad s\le\tau_R,
\]
and \(F(\{|z|>\delta_R\})<\infty\).  Consequently,
\[
 \mathbb E\int_0^{\tau_R}\int_{\mathbb R^q}|r_s(z)|F(dz)\,ds<\infty.
\]
The occupation estimate for \(X^R\) also gives
\[
 \mathbb E\int_0^{\tau_R}|Q^{jk}\partial_{jk}f(X_{s-})|\,ds
 \le C_{f_R}\|Q\|\|D^2f_R\|_p<\infty.
\]
The series \(\sum_{0<s\le\tau_R}\Theta_f(X_{s-},\Delta Y_s)\)
converges absolutely: the
terms with \(|\Delta Y_s|\le\delta_R\) agree with the corresponding terms
for \(f_R(X^R)\), and there are only finitely many jumps with
\(|\Delta Y_s|>\delta_R\).

For \(s\le\tau_R\), the local boundedness of \(Df\) and the
properties of \(h_f\) give, for a finite deterministic constant \(C_R\),
\[
 |h_f(\Delta_f(X_{s-},z))|^2
       \le C_R(1\wedge|z|^2),\qquad s\le\tau_R.
\]
Consequently,
\[
 \int_0^{\tau_R}\int_{\mathbb R^q}
       |h_f(\Delta_f(X_{s-},z))|^2F(dz)\,ds
       \le C_R\int_{\mathbb R^q}(1\wedge|z|^2)F(dz)<\infty,
\]
so the corresponding compensated jump integral is well defined.
Moreover, the choice of \(\delta_R\) ensures that
\[
 \Delta_f(X_{s-},z)-h_f(\Delta_f(X_{s-},z))=0,
       \qquad s\le\tau_R,\quad |z|\le\delta_R.
\]
Since \(F(\{|z|>\delta_R\})<\infty\), the uncompensated
integral is a finite sum of jumps on the stopped interval.
The continuous martingale integral is also well defined, since
\(Df(X_{s-})\sigma\) is bounded for \(s\le\tau_R\) and
\(\langle Y^c\rangle_t=tc\).

Substitute the canonical representation of \(Y\) relative to \(h\)
into this stopped It\^o formula.
The integrand \(\mathbf1_{\{s\le\tau_R\}}r_s(z)\) is predictable.
Its expected absolute compensator integral is finite by the bound
above, and the compensation formula gives the same for its
absolute jump integral. Consequently, componentwise, the compensated
integral equals the difference of the corresponding jump and
compensator integrals.
The algebra is
\[
 Df(X_{s-})\sigma\{z-h(z)\}+\Theta_f(X_{s-},z)
 =\Delta_f(X_{s-},z)-h_f(\Delta_f(X_{s-},z))+r_s(z).
\]
Together with \(Df(X_{s-})\sigma h(z)=h_f(\Delta_f(X_{s-},z))-r_s(z)\),
this yields \eqref{eq:app_augmented_fX_decomposition} and
\eqref{eq:levy_smoothing_beta}, stopped at \(\tau_R\), consistently
with \eqref{eq:app_augmented_drift_process}.
The first-order drift term is bounded before stopping, so the
preceding estimates also give \(\int_0^{\tau_R}|\beta_s|ds<\infty\)
almost surely.  Together with the absolute integrability of \(r\)
against the compensator and the absolute summability of the Taylor
remainders \(\Theta_f(X_{s-},\Delta Y_s)\) over \(0<s\le\tau_R\), this verifies
\eqref{eq:sm_common_calculus_integrability} after stopping.

Choose a Borel representative of the weak Hessian.  All terms in
\eqref{eq:levy_smoothing_beta} are predictable wherever the formula
is finite; assigning \(\beta=0\) on its predictable exceptional
set preserves predictability.  The stopped estimates and the
countable sequence of integers \(R>|x_0|\) show that this set is
\(ds\otimes d\mathbb P\)-null.
Finally, nonexplosion and the c\`adl\`ag paths of \(X\) imply that,
pathwise, \(\tau_R=1\) for all sufficiently large integers \(R\).
Removing the localization along this countable sequence proves
the assertion.
\end{proof}

\section*{Acknowledgments}
I would like to thank Masaaki Fukasawa and Constantinos Kardaras
for sharing their preliminary calculations on the \(n\)-scale
asymptotic error of the Euler scheme for a one-dimensional SDE
with additive Brownian noise. These calculations helped motivate
the present work.

\begingroup
% Keep the final entry with the bibliography; retain its font and line spacing.
\setlength{\bibitemsep}{3pt plus 1pt minus 0.5pt}
\setlength{\emergencystretch}{3em}
\printbibliography
\endgroup

\end{document}

%% file: settings_latex/set_text_en_default.tex
\usepackage{geometry} % document margins
\usepackage[T1]{fontenc} % font encoding
\usepackage{lmodern} % default english font % times

\usepackage{titlesec}
\titleformat*{\section}{\Large\bfseries} % section font (jlreq)
\titleformat*{\subsection}{\large\bfseries}
\titleformat*{\subsubsection}{\bfseries}

\usepackage{enumitem} % extending list environments % [label=\arabic*]
\setlist[itemize]{labelsep=1pt, partopsep=-3pt} % itemize

\setlist[enumerate]{labelsep=4pt, partopsep=-3pt} % enumerate

\setlist[description]{labelsep=6pt, partopsep=-6pt} % description

\usepackage{comment} % comment out multiple lines
\usepackage{array} % extending \array and \tabular
\usepackage{listings} % typeset source codes % \begin{lstlisting}
\usepackage[normalem]{ulem} % underlines % break lines only english
\usepackage{ascmac} % screen, itembox, shadebox
\usepackage{tcolorbox} % colored box % including graphicx, xcolor, tikz
\usepackage{subcaption}
\tcbuselibrary{breakable, skins, theorems} % tcolorbox library

\newcounter{rnum} % Don't use otf package if you use pdflatex

\newtcolorbox{simplebox}{
    colback=white, colframe=black, arc=2mm, boxrule=0.25mm,
    enlarge top by=2mm, enlarge bottom by=0mm, breakable=true
}

\newtcolorbox{titlebox}[1]{enhanced,
    colback=white, colframe=black, arc=2mm, boxrule=0.25mm,
    title={#1}, attach boxed title to top left={xshift=4mm,yshift=-2mm},
    colbacktitle=white, coltitle=black,boxed title style={boxrule=0.25mm},
    enlarge top by=2mm, enlarge bottom by=0mm, breakable=true, top=4mm
}

%% file: settings_latex/set_ref_clever_en_simple.tex
\usepackage{amsmath}
\usepackage{amssymb}
\usepackage{amsthm}
\usepackage{mathtools}
\usepackage{bm}

\numberwithin{equation}{section} % section number in equation number
\allowdisplaybreaks

\usepackage{hyperref}
\hypersetup{
    setpagesize=false, %bookmarksdepth=tocdepth, bookmarksnumbered=true,
    colorlinks=true, citecolor=blue, %green!75!black,
    linkcolor=blue, urlcolor=cyan!65!black,
}
\usepackage{cleveref}
\expandafter\def\csname ver@etex.sty\endcsname{3000/12/31} % get rid of the warning about etex

\usepackage{autonum}

\theoremstyle{plain} % modify \begin{theorem}[title]\label{}

\newtheorem{theorem}{Theorem}[section]
\newtheorem{proposition}[theorem]{Proposition}
\newtheorem{lemma}[theorem]{Lemma}
\newtheorem{corollary}[theorem]{Corollary}

\newtheorem{theorem*}{Theorem}[section]
\newtheorem{proposition*}[theorem]{Proposition}
\newtheorem{lemma*}[theorem]{Lemma}
\newtheorem{corollary*}[theorem]{Corollary}
\newtheorem{definition*}[theorem]{Definition}
\newtheorem{assumption*}[theorem]{Assumption}

\theoremstyle{remark}

\newtheorem{remark}[theorem]{Remark}

\newtheorem{remark*}[theorem]{Remark}
\newtheorem{example*}[theorem]{Example}

\crefname{theorem}{Theorem}{Theorems}
\crefname{proposition}{Proposition}{Propositions}
\crefname{lemma}{Lemma}{Lemmas}
\crefname{corollary}{Corollary}{Corollaries}
\crefname{definition}{Definition}{Definitions}
\crefname{assumption}{Assumption}{Assumptions}
\crefname{remark}{remark}{remarks}
\crefname{example}{example}{examples}

\crefname{equation}{Equation}{Equations}
\crefname{figure}{Figure}{Figures}
\crefname{table}{Table}{Tables}
\crefname{algorithm}{Algorithm}{Algorithm}

\crefname{section}{Section}{Sections}
\crefname{subsection}{Subsection}{Subsections}
\crefname{appendix}{Appendix}{Appendices}
\Crefname{appendix}{Appendix}{Appendices}

\ifdefined\AddToHook
  \AddToHook{env/theorem/begin}{\crefalias{theorem}{theorem}}
  \AddToHook{env/proposition/begin}{\crefalias{theorem}{proposition}}
  \AddToHook{env/lemma/begin}{\crefalias{theorem}{lemma}}
  \AddToHook{env/corollary/begin}{\crefalias{theorem}{corollary}}
  \AddToHook{env/definition/begin}{\crefalias{theorem}{definition}}
  \AddToHook{env/assumption/begin}{\crefalias{theorem}{assumption}}
  \AddToHook{env/remark/begin}{\crefalias{theorem}{remark}}
  \AddToHook{env/example/begin}{\crefalias{theorem}{example}}
  \AddToHook{cmd/appendix/after}{%
    \crefalias{section}{appendix}%
    \crefalias{subsection}{subappendix}%
    \crefalias{subsubsection}{subsubappendix}%
  }
\fi

%% file: settings_latex/set_math_symbol_simple.tex
\DeclarePairedDelimiterX{\inner}[2]{\langle}{\rangle}{#1, #2} 
\DeclarePairedDelimiterX{\Set}[2]{\lbrace}{\rbrace}{#1\,\delimsize\vert\,#2}

\DeclarePairedDelimiterXPP\borel[1]{\mathcal{B}}{\lparen}{\rparen}{}{#1}
\DeclarePairedDelimiterXPP\Prob[2]{#1}{\lparen}{\rparen}{}{#2}
\DeclarePairedDelimiterXPP\conProb[3]{#1}{\lparen}{\rparen}{}{#2\,\delimsize\vert\,#3}
\DeclarePairedDelimiterXPP\Exp[2]{\mathbb{E}_{#1}}{\lbrack}{\rbrack}{}{#2}
\DeclarePairedDelimiterXPP\Var[2]{\mathbb{V}_{#1}}{\lbrack}{\rbrack}{}{#2}
\DeclarePairedDelimiterXPP\Cov[3]{\mathbb{Cov}_{#1}}{\lbrack}{\rbrack}{}{#2, #3}
\DeclarePairedDelimiterXPP\conExp[3]{\mathbb{E}_{#1}}{\lbrack}{\rbrack}{}{#2\,\delimsize\vert\,#3}
\DeclarePairedDelimiterXPP\conVar[3]{\mathbb{V}_{#1}}{\lbrack}{\rbrack}{}{#2\,\delimsize\vert\,#3}

%% file: settings_latex/set_bib_default.tex
\cslet{blx@noerroretextools}\empty % to avoid incompatibility with autonum (etextool)
\usepackage[backend=biber,style=alphabetic]{biblatex}
\DeclareDelimFormat[bib]{nametitledelim}{\addcolon\space}

\renewbibmacro{in:}{}

\renewbibmacro*{volume+number+eid}{
    \printfield{volume}
    \setunit*{\addnbspace}
    \printfield{number}
    \setunit{\addcomma\space}
    \printfield{eid}
}
\DeclareFieldFormat[article]{volume}{\mkbibbold{#1}}
\DeclareFieldFormat[article]{number}{\mkbibparens{#1}}

\renewbibmacro*{issue+date}{}
\renewbibmacro*{note+pages}{
    \printfield{note}
    \setunit{\bibpagespunct}
    \printfield{pages}
    \setunit{\addcomma\addspace}
    \usebibmacro{date}
    \newunit
}
\DeclareFieldFormat{pages}{#1}
\DeclareFieldFormat{date}{\mkbibparens{#1}}

\DeclareFieldFormat{parenswithperiod}{\mkbibbrackets{#1\addperiod}}

\renewbibmacro*{pageref}{
    \iflistundef{pageref}{}{
        \printtext[parenswithperiod]{
            \ifnumgreater{\value{pageref}}{1}
                {\bibstring{backrefpages}\ppspace}
                {\bibstring{backrefpage}\ppspace}
            \printlist[pageref][-\value{listtotal}]{pageref}
        }
        
    }
}
\DefineBibliographyStrings{english}{
    backrefpage  = {\hspace{-0.8mm}cit.\,on page\hspace{-1.2mm}},
    backrefpages = {\hspace{-0.8mm}cit.\, on pages\hspace{-1.2mm}},
}
\DeclareCiteCommand{\citeauthor}{
    \boolfalse{citetracker}%
    \boolfalse{pagetracker}%
    \usebibmacro{prenote}
    }{\ifciteindex
        {\indexnames{labelname}}
        {}%
    \printtext[bibhyperref]{\printnames{labelname}\!}
    }{\multicitedelim}{\usebibmacro{postnote}
}
\DeclareCiteCommand{\citeauthoryear}{
    \boolfalse{citetracker}%
    \boolfalse{pagetracker}%
    \usebibmacro{prenote}
    }{\ifciteindex
        {\indexnames{labelname}}
        {}%
    \printtext[bibhyperref]{\printnames{labelname}}\printtext{\hspace{1mm}(\printfield{year})\hspace{-2mm}}
    }{\multicitedelim}{\usebibmacro{postnote}
}